\documentclass[a4paper,10pt]{article}
\usepackage{amsfonts}
\usepackage{mathrsfs}
\usepackage[breaklinks=true]{hyperref}
\usepackage{amsmath,amsthm,amscd,amssymb,graphicx}
\newtheorem{theorem}{Theorem}[section]
\newtheorem{lemma}[theorem]{Lemma}
\newtheorem{cor}[theorem]{Corollary}
\newtheorem{pro}[theorem]{Proposition}
\theoremstyle{definition}

\theoremstyle{remark}
\newtheorem{remark}[theorem]{Remark}

\numberwithin{equation}{section}

\newcommand{\norm}[1]{\lvert#1\rvert}
\newcommand{\mnorm}[1]{\lvert#1\rvert}
\newcommand{\ipro}[2]{\left\langle {#1},{#2}\right\rangle}

\newcommand{\Q}{\mathcal{Q}}
\newcommand{\J}{\textsc{J}}

\begin{document}
\title{Stable laws and spectral gap properties for affine random walks}
\author{Zhiqiang Gao\thanks{School of Mathematical Sciences, Beijing Normal University, Laboratory of Mathematics and Complex Systems, 100875 Beijing, China (e-mail: gaozq@bnu.edu.cn) } \and Yves Guivarc'h\thanks{Corresponding author, IRMAR, Universit\'e de Rennes-1, Campus de Beaulieu, 35042 Rennes Cedex, France (e-mail: yves. guivarch@univ-rennes1.fr). }  \and Emile Le Page\thanks{Laboratoire de Math\'ematiques de Bretagne Atlantique, UMR CNRS 6205, Universit\'{e} de Bretagne Sud,  Campus de
Tohannic, BP 573, 56017 Vannes,  France  (e-mail: emile.le-page@univ-ubs.fr). } }
\maketitle
\begin{abstract}
We consider a general multidimensional affine recursion with corresponding Markov operator $P$ and a unique $P$-stationary measure. We show spectral gap properties on H\"older spaces for the corresponding Fourier operators and we deduce convergence to stable laws for the Birkhoff sums along the recursion. The parameters of the stable laws are expressed in terms of basic quantities depending essentially on the matricial multiplicative part of $P$.
Spectral gap properties of $P$ and homogeneity at infinity of the $P$-stationary measure play an important role in the proofs.
 \end{abstract}
\section{Introduction and main results}

We consider the vector space $V= \mathbb{R}^d$ endowed with the scalar product
$\ipro{{x}}{{y}} = \sum_{i=1}^d x_iy_i$ and the norm $\norm{x}=\left(\sum_{i=1}^d|x_i|^2\right)^{1/2}$.
We denote by $H=V\rtimes G$ the affine group of $V$, with  $G=GL(d,\mathbb{R})$, i.e. the set of maps $h$ of the form $hx=gx+b (b\in V, g\in G)$.
Let $\mu$ be a probability measure on $H$ and $x\in V$.
We denote by $\mathbb{P}$ the product measure $\mu^{\otimes \mathbb{N}}$ on $\Omega=H^{ \mathbb{N}}$ and we consider the recurrence relation with random coefficients:
\begin{equation}\label{rwg1-1}
 X_0^x=x,\quad  X_n^x=M_n X_{n-1}^x+Q_n \quad (n\geq 1),
\end{equation}
where $(Q_n, M_n) \in H$ are i.i.d. random variables with generic copy $ (Q, M)$  and with law $\mu$.
Let $\bar{\mu}$ be the projection of $\mu$ on $G$, i.e. the law of $M$, and let $[\texttt{supp} \bar{\mu}]$ be the closed subsemigroup generated by the support of $\bar{\mu}$. We will denote by $P$ the corresponding Markov operator on $C_b(V)$, the space of continuous bounded functions on $V$:
\begin{equation*}
    P\varphi (x)=\int \varphi (gx+b) d\mu(h), \quad \varphi \in  C_b(V).
    \end{equation*}
    We observe that if $M_n=Id$ (resp $Q_n=0$), then $X_n^x$ is an additive (resp.
multiplicative) random walk on $V$ (resp. $V\backslash \{0\}$)(Cf \cite{Furstenberg72,GuivarchRaugi84,Levy54}).
Basic aspects of these special processes continue to hold in the general
case of $X_n^x$, and give a heuristic guide for the study of
the affine random walk $X_n^x$.
On the other hand,
independently of any density condition for $\mu$, the conjunction of these
two different processes give rise to new properties, in particular
spectral gap properties for $P$ (Cf \cite{BuraczewskiGuivarch10PTRF,GuivarchLePage08spectral}) and homogeneity at infinity
for the $P$-stationary measure(Cf \cite{BuraczewskiGuivarch09PTRF,Guivarch06DS,GuivarchLePage10}).

For a positive Radon measure $\rho$ on $V$ we denote $ \rho P$ the new measure obtained from $\rho$ by the dual action of $P$. Our hypothesis will imply that the above recursion \eqref{rwg1-1} has a unique stationary measure $\eta$ which satisfies $ \eta P= \eta$ and  has an unbounded support.
The probability measure $ \eta$ is the limit distribution of $X_n^x$. A remarkable property of $\eta $ is its ``homogeneity at infinity", a property which was first observed  in \cite{Kesten73Acta} for the tails of $\eta$, extended to the general case in \cite{LePage83}  and further developed in \cite{Alsmeyer10,BuraczewskiGuivarch09PTRF,Goldie91AAP}, under special  conditions. See  \cite{Guivarch06DS} for a survey of  \cite{LePage83} as well for a precise description of the homogeneity property of $\eta$, proved in a special case in \cite{BuraczewskiGuivarch09PTRF} and in a generic case in \cite{GuivarchLePage10}.

In this paper we are interested in the limit behavior of the sum $S_n^x= \sum_{k=0}^n X_k^x$, conveniently normalized. For $d=1$ this question is connected with the slow diffusion behavior of a simple random walk on $ \mathbb{Z}$ in a random medium (See \cite{KestenKozlovSpitzer75,Solomon75}).
The similar problem for a finitely
supported random walk on $\mathbb{Z}$ in a random medium is  connected  to  the study of a recurrence relation of  the form \eqref{rwg1-1} (See \cite{Goldsheid08PTRF, HongWang10}). More generally, the equation \eqref{rwg1-1} is of fundamental interest for
the study of generalized autoregressive processes( Cf \cite{BougerolPicard92,Kesten73Acta}). In
particular equation \eqref{rwg1-1} is a basic model in collective risk theory(\cite{Goldie91AAP}); in the context of extreme value theory, the corresponding convergence problem for normalized sample autocorrelations of a GARCH model is considered in \cite{Mikosch2000}.

For $d=1$,  and under aperiodicity conditions,  the limit behavior
of $S_n^x$ is described in \cite{GuivarchLePage08spectral}. For $d>1$, it turns out that, in the generic case considered below,
the limits are stable laws of general type and that the multiplicative part of the recursion plays a dominant role in the asymptotics. For
$d\geq 1$, in the case where $M_n $ takes values in the similarity
group of $V$, the limit behavior of $S_n^x$ is described in
\cite{BuraczewskiGuivarch10PTRF}; the homogeneity at infinity result
of \cite{BuraczewskiGuivarch09PTRF} plays an essential role in the
proof,  and \cite{BuraczewskiGuivarch10PTRF} contains a detailed description of
the limit laws  which turned out to be semi-stable in the sense of P. L\'{e}vy (See \cite[p.204]{Levy54}). For other situations where stable laws appear naturally in limits theorems in sums of non i.i.d random variables we refer to (\cite[p.321-323]{Levy54} ) and  \cite{BabillotPeigne06}.
Here we consider relation \eqref{rwg1-1} in the case where $[\texttt{supp}\bar{\mu}]$ is
``large",  a case which is generic and opposite to the case of
\cite{BuraczewskiGuivarch10PTRF}. We will need  the detailed
information on the stationary law $\eta$ of $P$ given in \cite{GuivarchLePage10} and summarised in Theorem \ref{rwgTh2-3} below;  also as in
\cite{BuraczewskiGuivarch10PTRF,GuivarchLePage08spectral}, a basic
role will be played by the spectral properties of the Fourier operators $P_v
(v\in \mathbb{R})$ defined by $P_v\varphi= P(\mathcal{X}_v \varphi )
$,  where $ \mathcal{X}_v(x)= e^{i\ipro{v}{x} }$.
Furthermore, the homogeneity at infinity of $\eta$  plays an essential role and implies that the
dominant eigenvalue of $P_v$ has an asymptotic expansion  at 0 in
terms of fractional powers of $\norm{v}$.  These properties allow us
to develop a detailed analysis and to prove limit theorems.  More generally, it
turns out that,  in the context of random walks associated with non abelian semigroup actions,
spectral gap properties are valid in certain functional
spaces for large classes of random walks. Usually, such properties are studied in the context of the so called ``Doeblin  condition"(See \cite{Alsmeyer10},\cite{DamekMMZ11} for example). Here instead, our study is based on the Ionescu-Tulcea and Marinescu theorem(\cite{Ionescu50AM}). This allows us to get spectral gap properties without density condition on $\mu$ or $\bar{\mu}$. See
\cite{CantatLeborgne05,ConzeGuivarch11,Dolgopyat02IJM,Goldsheid96PTRF,Furman99ETDS,Gouezel09DMJ,GuivarchHardy88AIHP}
for different classes of situations where analogous ideas are used.
Here $V$ can be considered as a boundary (see \cite{Furstenberg72})
for the random walk on $H$ defined by $\mu$, and we will use spectral gap
properties for $P_v (v\in V)$ in Banach spaces of H\"{o}lder functions with slow growth at infinity. In \cite{ConzeGuivarch11} and \cite{Furman99ETDS} the relevant
spaces are $L^2$-spaces, while in
\cite{CantatLeborgne05,Dolgopyat02IJM,Gouezel09DMJ}, they are of
mixed type. This type of analysis is not restricted to homogeneous
spaces of Lie groups  as shown in \cite{Mirek11PTRF} for certain
classes of Lipschitz maps instead of affine maps.
Here we follow the general line of \cite{GuivarchLePage08spectral,BuraczewskiGuivarch10PTRF}. With respect to these papers,
new arguments are needed
for the analysis of relation \eqref{rwg1-1}, in the generic case considered
below(See \cite{GuivarchLePage10}).

The asymptotics of products of random matrices (See
\cite{Guivarch08,Bougerol85RandomMatrices,GuivarchRaugi84}) will play an important
role, and we need to give corresponding notations. We say that
a semigroup $\Gamma\subset G$ is \emph{strongly irreducible }if no
finite union of proper subspaces of $V$ is $\Gamma$-invariant. Also
we say that $g\in G$ is \emph{proximal} if $g$ has a dominant
eigenvalue $\lambda (g) \in \mathbb{R}$ which is the unique
eigenvalue of $g$ such that $|\lambda(g)|=\lim_{n\rightarrow \infty}
\norm{g^n}^{1/n}$ where $\norm{g}= \sup\{ \norm{gx}: \norm{x}=1 \}
$. We say that $\Gamma$ satisfies condition \emph{i-p} if $\Gamma$
is strongly irreducible and contains a proximal element $\gamma$.
 It is proved in \cite{Pra94} that condition \emph{i-p} for $\Gamma$ and its Zariski closure $Zc(\Gamma)$ are equivalent.
Since  $Zc(\Gamma)$ is a closed Lie subgroup of $G$ with a finite number of connected components, condition  \emph{i-p} can be checked in examples
(see Section 5 for some examples).
 Under this condition, the limit set $L(\Gamma)\subset \mathbb{P}^{d-1} $ is the unique $\Gamma$-minimal subset of the projective space   $\mathbb{P}^{d-1} $  and $L(\Gamma )$ is the closure of the set of attracting fixed points of the proximal elements in $\Gamma$.

For $s\geq 0$, we denote $$\kappa(s)= \lim_{n \rightarrow\infty} (\mathbb{E} \mnorm{M_n\cdots M_1}^s)^{1/n},$$
$$s_\infty=\sup\{ s\geq 0 ; \kappa(s)<\infty\}.$$
For $g\in G$, we write $v(g)=\sup(|g|, |g^{-1}| )$. If $\mathbb{E} (\log v{(M)})<+\infty$, we know that the Lyapunov exponent $$L(\bar{\mu})= \lim_{n \rightarrow\infty} \frac{1}{n} \mathbb{E} (\log \mnorm{M_n\cdots M_1})$$ is well defined, $L(\bar{\mu})=\kappa'(0_+)$ if $s_\infty>0$.
If condition \emph{i-p} is satisfied and $s_\infty >0$, then $\log \kappa(s)$ is strictly convex on $[0,s_\infty)$, hence if $\lim\limits_{s\rightarrow s_\infty}\kappa(s)> 1 $, there exists a unique $\alpha \in (0,s_\infty)$ with $\kappa(\alpha)=1$.

  Our  hypothesis here is the following condition $C$ (See \cite{GuivarchLePage10}):
\\$C_1  \quad [\texttt{supp}\bar{\mu}]$  satisfies condition \emph{i-p},
\\$ C_2 \quad s_\infty>0,$ $ L(\bar{\mu})<0,$
             $\lim\limits_{s\rightarrow s_\infty}\kappa(s)> 1$,
\\ $C_3 \quad  \mathbb{E}(v{(M)}^{\alpha+\delta}+ \norm{Q}^{\alpha+\delta}) <\infty$  \mbox{ for some} $\delta >0$,
\\ $C_4 \quad \texttt{supp}\mu$ has no fixed point in $V$.

Condition $C$ will be assumed in our results (compare with condition ($H$) of \cite{BuraczewskiGuivarch10PTRF}), except if the contrary is specified. We observe that condition {\it i-p} for $[\texttt{supp}  ~\rho] $ is valid on an open dense set in weak topology of measures $\rho $  on $G$. It follows that condition $C$ is open in the weak topology of probability measures  on $H$.
Conditions $C_1$ and $C_3$ are used to prove homogeneity at
infinity of $\eta$,  a property which depends on the spectral gap properties of twisted
convolution operators defined by $\bar{\mu}$
on the projective space of $V$ (Cf \cite{GuivarchLePage10}). Condition $C_2$ plays the
basic role in the homogeneity at infinity of $\eta$.

A real number $t\in \mathbb{R}$ defines a dilation on $V$ which is  denoted by $v\rightarrow t.v$, and we extend this notation to the action of $\mathbb{R}$ on measures on $V$.
A Radon measure $\rho$ on $V$ is said to be $\alpha$-homogeneous if for any $t>0$, $t.\rho =t^\alpha \rho$.

Let $\overline{P}$ be the Markov operator on $V$ defined by
\begin{equation*}
    \overline{P}\varphi (v)= \int \varphi (gv) d\overline{\mu}(g), \quad \mbox{ if } \varphi \in C_b(V).
\end{equation*}
We observe that $\overline{P}$ can be interpreted as the linearisation of $P$ at infinity.
We denote by $\ell^s$ the $s$-homogeneous measure on $\mathbb{R}^*_+$ defined by $\ell^s(dt)= \frac{dt}{t^{s+1}}$.
It is proved in Theorem C of  \cite{GuivarchLePage10} that  if $d>1$ and condition $C$ is valid, there exists $c>0$ and a probability measure $\sigma_\alpha$ on the unit sphere $\mathbb{S}^{d-1}$ such that the following vague convergence is valid on $V \backslash\{0\}$:
\begin{equation}\label{rwg1-2}
    \lim_{t\rightarrow 0_+} t^{-\alpha} (t.\eta)= c \sigma_\alpha \otimes \ell^\alpha=
    \Lambda.
\end{equation}
Here $\Lambda$ is defined by the above convergence,
is $\alpha$-homogeneous,  and we have $
\Lambda \overline{P}= \Lambda$. We observe that the equation $\Lambda \overline{P}
= \Lambda$ is a limiting form of the stationarity equation
$\eta P=\eta$. The proof is based on the general renewal theorem of
\cite{Kesten74Ann} and on the spectral gap property of the
operator on the projective space defined by twisted convolution
with $\bar{\mu}$ (See \cite{GuivarchLePage04simple,GuivarchLePage10}).

More generally, if $\eta$ is a probability measure such that the above convergence \eqref{rwg1-2} is valid, we will say that $\eta$ is $\alpha$-homogeneous at infinity. 
A probability $\eta$ on $V$ is said to be stable if  for every integer $n$ there exists a
similarity $h_n$ of the form $h_n(x)= a_n x + b_n (a_n>0, b_n\in V)$ such that the $n^{th}$ convolution power of $\eta $ is the push forward of $\eta $ by
 $h_n$. If $a_n=n^{1/\alpha}$, we say that
$\eta$ is $\alpha$-stable.

Due to Theorem C of \cite{GuivarchLePage10}, if $\texttt{supp} \bar{\mu}$ has no invariant convex cone in $V$, then $\Lambda$ is symmetric and $ \sigma_\alpha\otimes \ell^\alpha$ is the  unique Radon measure defined by the following conditions:
 \begin{align*}
    &\sigma_\alpha \mbox{ is a probability measure on } \mathbb{S}^{d-1}, \\ &(\sigma_\alpha\otimes \ell^\alpha ) \overline{P}=\sigma_\alpha\otimes \ell^\alpha, \quad t.(\sigma_\alpha\otimes \ell^\alpha) = t^\alpha (\sigma_\alpha\otimes \ell^\alpha), \quad \mbox{for all } t>0.
\end{align*}
See \cite{GuivarchLePage10} for more detail. In Section 5 below we
give information on $\sigma_{\alpha}$ and examples of the typical
situations which can occur. In any case $\Lambda$ gives zero measure to any affine subspace, the projection of $ \sigma_\alpha $ on  the projective space
$\mathbb{P}^{d-1}$ is uniquely defined by the above condition and its support is equal to
the limit set $ L ([\texttt{supp}\overline{\mu}])$ in
$\mathbb{P}^{d-1}$.

We will write $g^*$ for the transposed map of $g\in G$, $\bar{\mu}^*$ for the push-forward of $\bar{\mu}$ by $g\rightarrow g^*$. Also for $x\in V$, we write $x^*$ for the linear form $x^*(y)=\ipro{x}{y}$. The exponential $e^{i\ipro{x}{y}}$ will be denoted by $\mathcal{X}_x(y)$ and the characteristic function of a probability measure $\pi$  on $V$ will be defined by
\begin{equation*}
\widehat{\pi}(x)=\int_V \mathcal{X}_x(y)d \pi(y).
\end{equation*}
 Coming back to the affine situation, we will write  $$m=\int x d \eta(x) ,\quad   m_\alpha= \kappa'(\alpha_-). $$

The calculation of the limit law of $S_n^x$ will involve considering the companion  recursion :
\begin{equation}\label{rwg1-3}
    W_0 =0, \quad W_n =M_n^* (W_{n-1}+ v),
\end{equation}
where $v\in V\backslash \{0\}$ is a fixed vector.
We will denote by $T_v$ the corresponding transition operator, i.e.
\begin{equation*}
    T_v(\varphi)(x)=\int \varphi(g^* (x+v)) d\bar{\mu}(g).
 \end{equation*}
Then as above, the unique stationary measure $\eta_v$ of $T_v$ satisfies the weak convergence on $V\backslash \{0\}$:
\begin{equation}\label{rwg1-4}
    \lim_{t\rightarrow 0_+} t^{-\alpha} (t.\eta_v)= \Delta_v\neq 0,
\end{equation}
and $\eta_v, \Delta_v$ satisfy
\begin{equation*}
    \eta_{tv}=t.\eta_v,\quad  \Delta_{tv}= t. \Delta_v \quad \mathrm{for} ~~~t\in \mathbb{R}^{\ast}, \quad   \Delta_v\overline{P}_* = \Delta_v, \quad \Delta_{tv}= t^{\alpha} \Delta_v  \mbox{ for } t>0,
\end{equation*}
where, as above,  $ \overline{P}_* $ is associated with $ \overline{\mu}^*$.

In order to state our first main result, we need to define a kind of Fourier transform $\widetilde{\Lambda}$  of $\Lambda$.
If $\alpha \in (0,2] $, we define $\widetilde{\Lambda}$ as follows:
\begin{eqnarray*}
   && \widetilde{\Lambda}(y)= \int \left(\mathcal{X}_y(x)-1\right)d \Lambda(x), \quad \mathrm{if} \quad 0<\alpha<1,\\
   &&  \widetilde{\Lambda} (y)=\int\left(\mathcal{X}_y(x) -1-i\frac{\ipro{x}{y}}{1+{|\ipro{x}{y}|^2}}\right) d\Lambda(x), \quad \mathrm{if} \quad \alpha =1,\\
   &&  \widetilde{\Lambda}(y)= \int\left(\mathcal{X}_y(x)-1-i\ipro{x}{y}\right) d \Lambda(x), \quad \mathrm{if} \quad 1<\alpha<2, \\
   &&  \widetilde{\Lambda}(y)= -\frac{1}{4} \int \ipro{y}{x}^2 d \sigma_2(x),\quad \mathrm{if} \quad \alpha=2.
\end{eqnarray*}

The function $\exp(\widetilde{\Lambda})$ is the Fourier transform of the limit law of the normalized sum of $\eta$-distributed
i.i.d random variables and $\widetilde{\Lambda}$ satisfies
\begin{equation*}
    \widetilde{\Lambda}(ty)= t^{\alpha}\widetilde{\Lambda}(y) \quad \mbox{for } t>0,  \quad \overline{P}_* \widetilde{\Lambda}=\widetilde{\Lambda},  \mbox{  and~~ } Re\widetilde{\Lambda}(y)<0  \mbox{ for }  y\neq 0. \end{equation*}


We will use also the function $\widetilde{\Lambda}^1$ defined by
 $     \widetilde{\Lambda}^1(y)=\widetilde{\Lambda}(\bar{y})\textbf{1}_{[1,\infty)}(\norm{y})$,
where $\overline{y}=y/\norm{y}$ denotes the projection of $y \in V\backslash \{0\}$ on $\mathbb{S}^{d-1}$.

The Fourier transform of the limit law of $S_n^x$ for $\alpha \in (0,2]$ will be shown to be equal to $e^{C_\alpha (v)}= \Phi_\alpha(v)$ where the function $C_\alpha(v) $ is  defined by
   \begin{equation}\label{rwg1-5}
    C_\alpha(v)= \left\{
                   \begin{array}{ll}
                     \alpha m_{\alpha} \Delta_v(\widetilde{\Lambda}^1) , & \mbox{ if } \quad \alpha\in(0,1)\bigcup (1,2]   ; \\
                        m_{1} \Delta_v(\widetilde{\Lambda}^1) + i \gamma (v), &\mbox{ if } \quad \alpha=1  ,
                   \end{array}
                 \right.
   \end{equation}
with
\begin{equation}\label{rwg1-6}
 \gamma(v)= \iint \Bigg[ \frac{\ipro{y+v}{x}}{{{1+|\ipro{y+v}{x}|^2}}}- \frac{\ipro{v}{x}}{ 1+{\norm{x}^2}}-\frac{\ipro{y}{x}}{1+{{|\ipro{y}{x}|^2}}} \Bigg] d\Lambda(x)d\eta_v(y).
\end{equation}
(See the proof of Proposition \ref{rwgpro4-2}.)  We have that for $ t>0$
$$ C_\alpha(tv)= t^\alpha C_\alpha(v) \quad  \mbox{  if  } \alpha \neq 1, \quad \mbox{ and } \quad C_1(tv) = tC_1(v) + i \ipro{v}{ \beta(t)}, $$
where $ \beta(t)= \int \left( \frac{tx}{1+\norm{tx}^2}-\frac{tx}{1+\norm{x}^2} \right)d \Lambda(x) $. Hence $ e^{C_\alpha(v)}$ is the Fourier transform of an infinitely divisible probability measure which belongs to an $ \alpha$-stable convolution semigroup (see \cite{IbraginovLinnik71Indpendent,JurekMason93,SamoTaqqu}).


If $\alpha >2 $, the following covariance  form $q$ of  $\eta$ will enter in the formulas below,
\begin{equation*}
    q(x,y)=\int \ipro{x}{\xi-m}\ipro{y}{\xi-m}d \eta(\xi).
\end{equation*}
We will write $z=\mathbb{E}(M) $ for the averaged operator of $M$ if $\alpha>1$. One sees easily that the operator $ \mathbb{E}M$ on $V$ exists and has spectral radius less than $\kappa(\alpha)=1$, hence  in particular $I-z^*$ is invertible.

We have the following limit theorem for the partial sums $S_n^x$.
\begin{theorem}\label{rwgTh1-5}
Assume that the probability measure $\mu$ on $H=V \rtimes G$ satisfies condition $C$ above. Then if $\dim V >1$,
 we have for any $x\in V$,
\\
1) If $\alpha>2$, $\frac{1}{\sqrt{n}} (S_n^x-n m)$ converges in law to the normal law on $V$ with the Fourier transform
\begin{equation*}
    \Phi_{2+}(v)=\exp(-q(v,v)/2 -q(v, (I-z^*)^{-1}z^*v)).
    \end{equation*}
2) If $\alpha \in (0,2)$, let $t_n=n^{-1/\alpha}$ and
\begin{equation*}
    d_n=\left\{ \begin{array}{ll}
                  0 ,& \quad \alpha\in (0,1); \\
                  n\delta(t_n), &\quad  \alpha=1 ;\\
                  nt_nm, &\quad  \alpha \in (1,2),
                \end{array}
    \right.     \quad
\end{equation*}
with    $\delta(t)=\displaystyle \int_V \frac{tx}{1+\norm{tx}^2}d\eta(x)$   for $ t>0.$
Then $(t_n S_n^x-d_n)$ converges in law to the $\alpha$-stable law with the Fourier transform $\Phi_{\alpha} (v)=\exp(C_\alpha (v))$, with $C_\alpha (v) $ given above.

Furthermore if $\alpha=1$, then for some constant $K_\star>0$,
\begin{equation*}
  |\delta(t)| \leq \left\{ \begin{array}{ll}
                           K_\star |t||\log |t|| , & \quad  for ~~~|t|\leq \frac{1}{2}; \\
                          K_\star |t|, & \quad   for ~~~|t|>\frac{1}{2}.
                         \end{array}
    \right.
\end{equation*}
3) If $\alpha=2$, then $\displaystyle \frac{1}{\sqrt{n\log n}}(S_n^x-nm)$ converges in law to the normal law with Fourier transform
$$    \Phi_{2}(v)=\exp (C_2(v)), \quad \mbox{ where } \quad C_2(v)= -\frac{1}{4}\int (\ipro{v}{w})^2+2\ipro{v}{w} \eta_v(w^*)d \sigma_2(w).$$
4)In all cases, the limit laws are fully non degenerate.
\end{theorem}

The proof of Theorem \ref{rwgTh1-5} is based on the method of  characteristic functions. The characteristic function of $S_n^x$ can be expressed in terms of iterates of the  Fourier operator $P_v$ defined above.
 This operator  acts as a bounded operator on a certain Banach space $\mathbb{B}_{\theta, \varepsilon,\lambda}$ (defined below) of unbounded functions on $V$ and has ``nice" spectral properties on $\mathbb{B}_{\theta, \varepsilon,\lambda}$. Moreover $P_0=P $ and the spectral properties of $P_v$ allow to control the perturbation  $P_v$ of $P$ as well as its dominant eigenvalue $k(v)$. Theorem \ref{rwgTh1-5}   follows from the asymptotic expansion of $k(v)$ at $v=0$, which is based on the homogeneity at infinity of $\eta$ and $\eta_v$. The spectral properties of $P_v$ follow from a theorem of Ionescu-Tulcea and Marinescu based on certain  functional inequalities proved below  which are consequences of the condition $L(\bar{\mu})<0$.

 We denote by $r(U)$  the spectral radius  of a bounded linear operator $U$. The spectral properties of $P_v$ are described by  the:
 \begin{theorem}\label{rwgTh1-6}
 If $v\in V$, the operator $P_v$ on $\mathbb{B}_{\theta,\varepsilon,\lambda}$ defined by $P_v f =P(\mathcal{X}_v f )$ has the following properties:
 \\ 1) $P_v$ is a bounded operator with spectral radius at most 1,
 \\ 2) If $v\neq 0$, $r(P_v)<1$,
 \\ 3) If $v=0$  and $\pi_0$ is the projection on $\mathbb{C}\mathbf{1}$ defined by  $ \pi_0 \varphi = \eta(\varphi)\mathbf{1}$, we have for any $ \varphi \in \mathbb{B}_{\theta,\varepsilon,\lambda}$ : $$P_0\varphi =  \pi_0 \varphi+ \mathcal{Q}\varphi$$
 where $\mathcal{Q}\pi_0=\pi_0 \mathcal{Q}=0$ and $r(\mathcal{Q})<1$.
 \\ 4) If $v$ is small, $P_v$ has a unique eigenvalue $k(v)$ with $|k(v)|= r(P_v)$. Furthermore  there exists a one dimensional projection $\pi_v$ and a bounded operator $\mathcal{Q}_v$ such that $ \mathcal{Q}_v \pi_v=\pi_v \mathcal{Q}_v=0$, $r(\mathcal{Q}_v)<|k(v)|$  and
 $$P_v \varphi= k(v) \pi_v \varphi+ \mathcal{Q}_v \varphi, \quad   for~~~ any ~~~ \varphi\in \mathbb{B}_{\theta,\varepsilon,\lambda}.$$
  Furthermore $k(v), \pi_v, \mathcal{Q}_v$ depend continuously on $v$.
  \end{theorem}
  These spectral properties will allow us to reduce the study of the iterated operator $P_v^n$ to the study of  its dominant eigenvalue $k^n(v)$; hence $k(v)$ plays here the role of  a characteristic function for the convolution operator $P$ defined by $\mu$ on $C_b(V)$.

The asymptotic behavior  of $k(v)$ at $v=0$ is given by the
\begin{theorem}\label{rwgTh1-3}
Let $v\in V\backslash \{0\}$ and let $C_\alpha(v) $ be given by \eqref{rwg1-5}.
\\
1) If $0<\alpha<1$, then $$\lim_{t\rightarrow 0_+} \frac{k(tv)-1}{t^{\alpha}}=C_{\alpha}(v).$$
\\ 2)If $\alpha =1 $, then
\begin{equation*}
    \lim_{t\rightarrow 0_+} \frac{k(tv)-1-i\ipro{v}{\delta(t)}}{t}=C_1(v).
\end{equation*}
\\ 3)If $1<\alpha<2$, then
$$ \lim_{t\rightarrow 0_+}\frac{k(tv)-1-i\ipro{v}{tm}}{t^\alpha} =C_\alpha (v).$$
\\ 4)If $\alpha=2$, then
\begin{equation*}
  \lim_{t\rightarrow 0} \frac{k(tv)-1-i\ipro{v}{tm}}{t^2|\log|t||} =2C_2(v).
\end{equation*}
\\ 5)  If $\alpha>2$, then
\begin{equation*}
    \lim_{t\rightarrow 0} \frac{k(tv)-1-i \ipro{v}{tm}}{t^2}= C_{2+}(v),
\end{equation*}
with $$C_{2+}(v)=-\frac{1}{2} q(v,v)- q(v, (I-z^*)^{-1} z^*v) .$$
\end{theorem}
As in \cite{GuivarchLePage08spectral} and \cite{BuraczewskiGuivarch10PTRF}, the proof of Theorem \ref{rwgTh1-3} is based on an intertwining  relation between the families of operators $P_v$ and $T_v$  and on the homogeneity   at infinity of $\eta$, $ \eta_v$ proved in \cite{GuivarchLePage10};  this relation allows us to express $k(v)$ in terms of the stationary measure $\eta$ and an eigenfunctional  for $T_v$.
\begin{remark}~~\\
a) We may observe that, if we add stronger moment conditions (of order greater than 4), part 1 of Theorem \ref{rwgTh1-5}, i.e. convergence to a normal law, follows from  the main result of \cite{HennionHerve04AP}, which is valid also  for more general Lipschitz  maps of $V$ into itself.
\\ b) For $\alpha \in [0,2]$, the limit law of $S_n^x$ is a multidimensional $\alpha$-stable law (see e.g.  \cite{IbraginovLinnik71Indpendent,JurekMason93,Levy54}) where $\alpha$-stability   holds with respect to  the action of the dilation group $\mathbb{R}^*_+$. In  particular  the limit  law is infinitely divisible and  belongs to a  convolution semigroup  of $\mathbb{R}^d$.  This remarkable fact follows from the homogeneity of $\Delta_v$ with respect to $v$, hence from the formula for $C_{\alpha}(v)$. \\
c) It follows from Theorem \ref{rwgTh2-3} below that the negative definite function  $C_\alpha$ satisfies $ Re C_{\alpha}(v)<0$ for $v$ non zero. In section 5 below, we obtain more detailed information on the function $C_\alpha$.  In particular, the function  $C_\alpha$ depends continuously on $\mu$ in a natural weak topology which guarantees continuity of moments of order $\alpha$. Also, given $\bar{\mu}$, the magnitude of $C_\alpha$ is closely related to the magnitude of
the moment  of order  $\alpha$ for $Q$. It follows that, for the stable limiting laws of the theorem, various situations occur, as in the case of sums of $\eta$-distributed i.i.d random variables on $V$: symmetric, non symmetric, supported on a proper convex cone.
\\  d) The fact that the stability group here is $\mathbb{R}^*_+$,  if $\alpha $ belongs to $[0,2]$ instead of a  more complex one as in \cite{BuraczewskiGuivarch10PTRF}, is a  consequence of the following property depending on condition $i$-$p$ and  $d\geq 2$ (see \cite{GuivarchRaugi84,GuivarchUrban05SM}): the closed subsemigroup of $ \mathbb{R}^*_+$  generated by  the moduli of the dominant eigenvalues  for  the proximal elements in $[\texttt{supp} \bar{\mu}]$ is equal  to $ \mathbb{R}^*_+$. This can be compared with the situation of \cite{BuraczewskiGuivarch10PTRF} where semi-stable laws in the sense of \cite[p.204]{Levy54} appear as limits.
As already mentioned Condition $C$ is generically satisfied by $\mu$, and  like in the case $ \alpha >2$ of the main theorem in \cite{BuraczewskiGuivarch10PTRF},  our limit theorem is essentially not changed under perturbation of $\mu$. This open the possibility of getting convergence to stable laws in natural multidimensional stochastic systems.
\\ e) The theorem gives the convergence of  normalized 1-marginals of $S_n^x$. A natural question is the existence of a functional limit theorem, i.e. the convergence towards a stable stochastic  process with continuous time (Cf \cite{Levy54,SamoTaqqu}).

\end{remark}
We note that closely related limit theorems for $S_n^x$ have been obtained recently in
the reference \cite{DamekMMZ11}, under a stronger hypothesis than here. In \cite{DamekMMZ11}, $\bar{\mu}$ dominates a density on $G$ and $[\texttt{supp} \bar{\mu}]$ has no invariant convex cone, hence the limiting law is symmetric. Furthermore $\alpha=2$ is excluded  and
 the case $\alpha=1$  is treated under symmetry restrictions. The method is based on a renewal theorem of \cite{Alsmeyer10} for a Markov chain which satisfies Harris condition.

\section{Homogeneity at infinity of $\mu$-stationary measures}
The following proposition gives the existence and elementary properties of the stationary law of $X_n^x$ in our context. The first part is well known.
\begin{pro}\label{rwgpro2-2} Assume that $\mu$ satisfies condition $C$. Let $$R_n =Q_1 +\sum_{k=1}^{n-1}M_1\cdots M_{k}Q_{k+1}.$$
Then $R_n$ converges a.e. to $$R= Q_1+ \sum_{k=1}^\infty M_1\dots M_k Q_{k+1}$$ and the law of $X_n^x$ converges to the law $\eta$ of $R$.
Furthermore, $\eta $ has no atom, gives measure zero to every affine subspace and $\mathbb{E} (\norm{R}^{\theta})=\int \norm{x}^{\theta} d\eta(x) <\infty$
if $\theta<\alpha$.
\end{pro}
 \begin{proof}
 The proofs of convergence are  based on known arguments (see \cite{BougerolPicard92,Kesten73Acta}), hence we give only a sketch  in our setting. If $s<\alpha$, we have by definition of $\kappa(s)$:
 \begin{equation*}
    \mathbb{E}(\mnorm{M_1\cdots M_k}^s)= \mathbb{E}(\mnorm{M_k\cdots M_1}^s) \leq C(\kappa(s)+\epsilon)^k
 \end{equation*}
  for some $C>0$, any integer $k>0$ and $0<\epsilon< \kappa(\alpha) -\kappa(s)$.
 Also $\mathbb{E}(\norm{Q_k}^s)= \mathbb{E}(\norm{Q_1}^s) \leq \mathbb{E}(\norm{Q_1}^\alpha)^{s/\alpha} <\infty $.
It follows if $m>n$,
$$\mathbb{E} (\norm{R_m-R_n }^s) \leq C (\mathbb{E}(\norm{Q_1}^{\alpha}))^{s/\alpha} \sum_{k=n}^{m-1} (\kappa(s)+ \epsilon)^k  <\infty.$$
Hence $\lim_{m,n\rightarrow \infty} \mathbb{E} (\norm{R_m-R_n}^s)=0 $. The  convergence  a.e. of  $R_n $ to $R$ follows.
The same calculation shows $\mathbb{E}(\norm{R}^{\theta})<\infty$ if $\alpha \leq 1$ and $\theta <\alpha$.
 If $\alpha >1$ and $\theta \in [1,\alpha[$, we use Minkowski inequality in $\mathbb{L}^{\theta}(\Omega)$ and the independence of $M_1\cdots M_{k-1}, Q_k$ to get that :
  \begin{equation*}
    \mathbb{E}(\norm{R}^{\theta}) \leq C\mathbb{E}(\norm{Q_1}^{\theta})\left[1+ \sum_{k=1}^{\infty} (\kappa(\theta)+\epsilon)^{k/\theta} \right]^{\theta}<\infty,
 \end{equation*}
 if $ \epsilon$ satisfies $\kappa(\theta)+\epsilon <1$.

 The fact that $\eta$ has no atom is proved as follows.

 Let $A\subset V$ be the set of atoms of $\eta$. Then $A$ is countable and $\sum_{x\in A} \eta(\{x\}) \leq 1$.
 It follows that, for every $\epsilon>0$, the set $\{ x\in A; \eta(x)\geq \epsilon\}$ is finite; in particular,  $\sup_{x\in A} \eta(\{x\})=c$ is attained. Let $ A_0= \{x\in A ; \eta(x)=c\}$. Since $\eta P=\eta$, we have $h A_0=A_0 $ if $h\in \texttt{supp}\mu.$
 Then the barycenter of $A_0$ is a \texttt{supp}$\mu$-invariant point, which is excluded by condition $C_4$.

  Assume now that there exists an affine subspace $W$ of positive dimension such that $\eta (W)>0$, and let $\mathcal{W}$ be the set of affine subspaces of minimum dimension $r$ with $\eta(W)>0$. If $r=0$, the contradiction follows from above. If $r>0$, we observe that for any $W, W' \in \mathcal{W}$ with $W\neq W'$, we have $ \eta(W\bigcap W') =0$  since $\dim(W\bigcap W') <\dim W$. Then as above $\sup_{W\in\mathcal{W}} \eta (W)=c'$ is attained. If $ \mathcal{W}_0= \{W\in \mathcal{W}: \eta(W)= c'\}$, we have $h\mathcal{W}_0=\mathcal{W}_0 $ for any $h\in \texttt{supp}\mu$. Let $\Gamma$ be the closed subgroup of $H$ generated by $\texttt{supp}\mu$, hence  $h\mathcal{W}_0=\mathcal{W}_0 $ for any $h\in \Gamma$.
Then the subset $\Gamma_0$ of $\Gamma$, which leaves invariant any $W\in \mathcal{W}_0$, is a finite index subgroup of $\Gamma$.
Since $L(\bar{\mu})<0$, $[\texttt{supp}\bar{\mu}]$ has an element $g$ with $\norm{g}<1$. Assume $h\in [\texttt{supp}{\mu}] $ has linear part $g$ and observe that $h$ has a unique fixed point $x\in V$ which is attracting. Since $\Gamma_0$ has finite index in $\Gamma$, we can find $p\in \mathbb{N}$ such that $h^p \in \Gamma_0$.
Then for any $y\in W$ with $W\in \mathcal{W}_0$, we have
$$\lim_{n\rightarrow \infty} h^{pn}y =x.$$
 Since $h^{pn}y \in W$, we get $x\in W$, hence $$x\in \bigcap_{W\in\mathcal{W}_0} W \neq \varnothing .$$
 It follows that $\Gamma$ leaves invariant the nontrivial affine subspace $ \bigcap_{W\in\mathcal{W}_0} W $. If $\dim \bigcap_{W\in\mathcal{W}_0} W =0,$
 we have constructed a point invariant under $\Gamma$, which contradicts conditions $C_4$. If $\dim \bigcap_{W\in\mathcal{W}_0} W >0,$ the direction of this affine subspace is a proper $\texttt{supp}\bar{\mu}$-invariant linear subspace, which contradicts condition \emph{i-p} for $\texttt{supp}\bar{\mu}$.
 \end{proof}
For $\kappa(s)$ we have the following proposition(see \cite{GuivarchLePage04simple}):
  \begin{pro}\label{rwgpro2-1}
  Assume $[\texttt{supp} \bar{\mu}]$ satisfies conditions i-p. Then $\log \kappa(s) $ is strictly convex on $[0,s_\infty[$.
  If $s_\infty=\infty$, we have:\begin{eqnarray*}
                                  \lim_{s\rightarrow \infty} \frac{\log \kappa (s)}{s} &=& \lim_{n\rightarrow \infty} \frac{1}{n}
                                   \sup \{ \log \mnorm{g}: g\in [\texttt{supp}\bar{\mu}]^n\} \\
                                   &=&   \lim_{n\rightarrow \infty} \frac{1}{n} \sup\{ \log r(g): g\in [\texttt{supp}\bar{\mu}]^n \}
                                \end{eqnarray*}
In particular, the condition $\kappa(s)<1$ on $]0,\infty[$ is equivalent to $r(g)\leq 1 $ on $[\texttt{supp} \bar{\mu}]$, and
if $\lim_{s\rightarrow s_\infty} \kappa (s) \geq1$  there exists a unique $\alpha \in ]0,s_\infty]$ such that $\kappa(\alpha)=1$.
\end{pro}
\begin{remark}
Regularity properties of $\kappa(s)$, not used here, are proved in \cite{GuivarchLePage04simple}. In particular, $\kappa(s)$ is analytic on $[0,s_{\infty}[$.
\end{remark}

It is known (see \cite{GuivarchLePage04simple,GuivarchLePage10}) that since  $\bar{\mu}$ satisfies condition  \emph{i-p} and $\kappa(s)<\infty$, there exists a unique probability measure $\nu_s$ on $\mathbb{P}^{d-1}$ such that the $s$-homogeneous Radon measure $\nu_s\otimes \ell^s$ on $\mathbb{P}^{d-1}\times \mathbb{R}^*_+= (V\backslash \{0\})/ \{\pm Id\}$ satisfies
\begin{equation*}
(\nu_s\otimes \ell^s) \bar{P}= \kappa(s) \nu_s\otimes \ell^s,
\end{equation*}
where, by abuse of notation, $ \overline{P} $ is the Markov operator defined  by  $ \overline{\mu}$ on $(V\backslash \{0\})/ \{\pm Id\}$. If $\bar{x}\in \mathbb{P}^{d-1}$ corresponds to $x\in V$, we denote $\norm{g \bar{x}}= \frac{\norm{gx}}{\norm{x}}$ and we consider the operator $\rho_s(\bar{\mu})$ on $\mathbb{P}^{d-1} $ defined by
\begin{equation*}
    \rho_s(\bar{\mu})(\varphi)(\bar{x})= \int \varphi(g\cdot \bar{x})\norm{g\bar{x}}^s d\bar{\mu}(g),
\end{equation*}
where $\bar{x} \mapsto g\cdot \bar{x}$ is the projective map defined by $g\in G$. Then $\nu_s$ is the unique probability measure on $\mathbb{P}^{d-1}$ such  that $\rho_s(\bar{\mu})\nu_s= \kappa(s) \nu_s$.
 Furthermore, $\texttt{supp} \nu_s$ is equal to the limit set of $[\texttt{supp} \bar{\mu}]$ and $\nu_s$ gives zero measure to any projective subspace(see \cite{GuivarchLePage04simple,GuivarchLePage10}).  In the corresponding situation for the unit  sphere, either there exists a unique probability measure $\sigma_s$ on the unit  sphere which satisfies the above equation or there exist two such measures with disjoint supports which are extremal and symmetric to each other (see  \cite[Theorem 2.17]{GuivarchLePage10}), if $[\texttt{supp} \bar{\mu}]$ preserves a convex cone.  The following consequence  of the general renewal theorem of \cite{Kesten74Ann} and of the spectral gap property of the operator $\rho_s (\bar{\mu})$ is proved in  \cite[Theorem C]{GuivarchLePage10}  and plays an essential  role here.
\begin{theorem}\label{rwgTh2-3}
If $d>1$ and  condition $C$ holds, we have the following weak convergence:
\begin{equation*}
    \lim_{t\rightarrow 0_+} t^{-\alpha}(t.\eta)= c(\sigma_\alpha \otimes \ell^\alpha)= \Lambda,
\end{equation*}
where $c>0$, $\sigma_\alpha$ is a probability measure on $\mathbb{S}^{d-1}$ which has projection $\nu_\alpha$ on $\mathbb{P}^{d-1}$ and $\Lambda$ satisfies $t. \Lambda = t^\alpha \Lambda$  if $t>0$, $\Lambda \bar{P}= \Lambda$.
The above convergence is valid for any function $f$ with a $\Lambda$-negligible set of discontinuities and  such that for some $\varepsilon>0$
\begin{equation}\label{rwg2-1}
 \sup_{x\neq 0}\left( \norm{x}^{-\alpha} |\log \norm{x}|^{1+\varepsilon} |f(x)|\right)<\infty.
\end{equation}
In particular there exists $A>0$ such that for $k$ large enough,
\begin{equation}\label{rwg2-1-a}
\frac{1}{A} 2^{-k\alpha}\leq  \eta \{ x\in V ; \norm{x}\geq 2^k \}\leq A 2^{-k\alpha} .
\end{equation}
Also  $\Lambda(W)=0$ for any proper  affine subspace $W\subset V$.
\end{theorem}
In the special case of the recurrence relation $$W_n=M_n^*(W_{n-1} +v)\quad  (n\geq 1),$$
the corresponding measure on $H$ is denoted by ${\mu}^*_v$.    
The corresponding transition operator on $V$ is denoted by $T_v$. Then we have the
\begin{pro}\label{rwgpro2-5}
Assume Condition $C$  holds true for $\mu$. Then Condition $C$ is satisfied by the measure ${\mu}^*_v$ on $H$, if $v\neq 0$.

 The sequence $$Z_n^*= \sum_{k=1}^{n} M_1^*\cdots M_k^*$$
converges $\mathbb{P}$-a.e. to $$Z^*=\sum_{k=1}^{\infty} M_1^*\cdots M_k^*$$ where $Z$ is defined by the   $\mathbb{P}$-a.e
convergent series $\sum_{k=1}^{\infty} M_k\cdots M_1$.

 The law $\eta_v$ of $Z^*v$ is the unique ${\mu}^*_v$-stationary measure and $\eta_v$ satisfies
 $$\int \norm{x}^\theta d\eta_v(x) <\infty\quad for \quad  \theta \in [0,\alpha[, \quad  \int \norm{x}^\alpha d\eta_v(x)=\infty.$$

 For any $t\in \mathbb{R}^* $, we have  $\eta_{tv}=t.\eta_v$. If $\alpha >1$,   for all $x\in V$ the map $v\rightarrow \eta_v(x^*)$  is a linear form.

 The Radon measure $$ \Delta_v= \lim_{t\rightarrow 0_+} t^{-\alpha} (t.\eta_v) $$
 is $\alpha$-homogeneous, satisfies $\Delta_{tv}=t^\alpha\Delta_v$ for $t>0$,  $ \Delta_v\overline{P}_*=\Delta_v $, $\Delta_{-v} $ is symmetric of $ \Delta_v$.

 The function $C_{\alpha}(v)$  satisfies for $v\neq 0$, $Re C_{\alpha}(v) <0$  and for $ t>0$,
$$ C_\alpha(tv)= t^\alpha C_\alpha(v) \quad  \mbox{  if  } \alpha \neq 1, \quad \mbox{ and } \quad C_1(tv) = tC_1(v) + i \ipro{v}{ \beta(t)}, $$
where $ \beta(t)= \int \left( \frac{tx}{1+\norm{tx}^2}-\frac{tx}{1+\norm{x}^2} \right)d \Lambda(x) $.
 \end{pro}
\begin{proof}
We observe that $\mnorm{M^*}=\mnorm{M}$, hence $$\lim_{n\rightarrow \infty} (\mathbb{E} (\mnorm{M_n^*\cdots M_1^*}^s))^{1/n}= \lim_{n\rightarrow \infty}
(\mathbb{E} (\mnorm{M_1\cdots M_n}^s))^{1/n}= \kappa (s).$$
One verifies easily that condition  \emph{i-p} for $[\texttt{supp} \bar{\mu}]$, which is valid,  remains valid for $ [\texttt{supp} \bar{\mu}]^*=[\texttt{supp} \bar{\mu}^*]$.
If $\texttt{supp} \bar{\mu}^*$ had a fixed point $x\in V$, then $g^* (x+v)=x$ for any $g\in \texttt{supp}\bar{\mu} $.
Since $v$ is non zero, we have  $x\neq 0$. Also this implies $g_1^*(g_2^*)^{-1}x=x$ for any $g_1,g_2 \in \texttt{supp} \bar{\mu}$, hence $x$ is invariant under the subgroup generated by $\texttt{supp} \bar{\mu}$. This contradicts irreducibility of  $ [\texttt{supp} \bar{\mu}]$.

 As in the proof of proposition \ref{rwgpro2-2}, one sees that the condition
$$\lim_{n\rightarrow\infty} (\mathbb{E} (\mnorm{M_1^*\cdots M_n^*}^\theta))^{1/n}=\kappa(\theta)<1$$
 for $\theta <\alpha$ implies the convergence  $$\lim_{n\rightarrow \infty}  \sum_{k=1}^{n}M_1^*\cdots M_k^* =\sum_{k=1}^{\infty}M_1^*\cdots M_k^*=Z^* .$$
 Since the map $g\rightarrow g^*$ is continuous, this gives the convergence of $Z_n=\sum_{k=1}^{n} M_k\cdots M_1$ to $Z=\sum_{k=1}^{\infty} M_k\cdots M_1$.

The second assertion on $\eta_v$ follows from inequality \eqref{rwg2-1-a} of Theorem \ref{rwgTh2-3} applied to $\mu^*_{v}$, since  Proposition \ref{rwgpro2-5} implies that  Condition C for $\mu$ and $\mu^*_v$ are equivalent.

The third assertion on linearity of $\eta_v$ with respect to $v$ follows from the relations $$Z^*(tv)= t Z^*(v), Z^*(v+w)=Z^*(v)+Z^*(w).$$

The last assertions  follow  from Theorem \ref{rwgTh2-3}, the relation $ \eta_{tv}=t.\eta_v$ for $t\in \mathbb{R}^*$ and the definition of $C_{\alpha}(v)$.
\end{proof}

We recall that the characteristic function $ \widehat{\eta_v} $ of the measure $\eta_v$  is  defined by $ \widehat{\eta_v}(x)= \eta_v(\mathcal{X}_x)$  and $w^*=\langle w,\cdot \rangle.$  In the proof of Theorem \ref{rwgTh1-3}, we shall need the following formula for the quantity ${C}_{\alpha}(v)$. We denote by $\widehat{C}_{\alpha}(v)$ the following quantity:
%
  \begin{equation}\label{rwg2-2}
     \widehat{C}_\alpha(v)= \left\{\displaystyle
                    \begin{array}{ll}\displaystyle
                      \int (\mathcal{X}_{v}(x)-1 ) \widehat{\eta_v}(x) d \Lambda (x) , &  \mbox{ if } 0<\alpha< 1; \\
                    \displaystyle \int\left((\mathcal{X}_v(x)-1)\widehat{\eta_v}(x)-i\frac{\ipro{v}{x}}{{1+\norm{x}^2}}\right) d \Lambda (x)  , &  \mbox{ if } \alpha= 1;  \\
                     \displaystyle \int \left( (\mathcal{X}_v(x)-1)\widehat{\eta_v}(x)- i\ipro{v}{x}  \right) d \Lambda (x), & \mbox{ if } 1<\alpha< 2; \\
                    \displaystyle -\frac{1}{4} \int (\ipro{v}{w}^2 +2\ipro{v}{w}\eta_v(w^*)) d \sigma_2(w)  , &  \mbox{ if }  \alpha= 2
                    \end{array}
                  \right.
  \end{equation}

\begin{pro}\label{rwgpro4-2}
The  formula   $\widehat{C}_{\alpha}(v)$ $=$ $C_\alpha(v)$  with the  definition \eqref{rwg1-5} is valid.
\end{pro}
\begin{proof}
We start as in the proof of Proposition 5.19 in \cite{BuraczewskiGuivarch10PTRF}.
By definition of $\widetilde{\Lambda}$, we have
\begin{equation*}
   \widehat{C}_\alpha(v)=\left\{
                  \begin{array}{ll}\displaystyle
                   \int\left( \widetilde{\Lambda}(y+v)-\widetilde{\Lambda}(y)\right)d\eta_v(y)  , & \mbox{ if } \quad \alpha\in(0,1)\bigcup (1,2]  ;\\
               \displaystyle  \int\left( \widetilde{\Lambda}(y+v)-\widetilde{\Lambda}(y)\right)d\eta_v(y)      + i \gamma (v), &\mbox{ if } \quad \alpha=1,
                  \end{array}
                \right.
\end{equation*}
where $\gamma(v)$ is given by  \eqref{rwg1-6}.
We follow the argument in \cite{BuraczewskiGuivarch10PTRF}, but we use in an essential way the information of \cite{GuivarchLePage10}(See Theorems 2.6, 2.17), and in particular Theorem \ref{rwgTh2-3} above.

We define for $s<\alpha $ the Radon measure $\Lambda_s$ by $$\Lambda_s=c \sigma_s\otimes \ell^s,$$ where  $c$ is given by Theorem \ref{rwgTh2-3} and $\sigma_s$ is a probability measure on $\mathbb{S}^{d-1}$, depending continuously on $s$ in weak topology, such that $$\Lambda_s\bar{P}= \kappa(s) \Lambda_s, \quad \mbox{and} \quad \lim_{s\rightarrow \alpha_-} \sigma_s= \sigma_\alpha,$$
and $ \sigma_\alpha$ given by Theorem \ref{rwgTh2-3}.
 The existence and continuity of $\sigma_s$ for $s<\alpha$ follow from  the discussion of stationary  measures given before Theorem \ref{rwgTh2-3}, which is based on  (\cite{GuivarchLePage10}, Theorem 2.17). Hence we have the weak convergence:
$$ \lim_{s\rightarrow \alpha_-} \Lambda_s= \Lambda_{\alpha}=\Lambda.$$
We define also $\widetilde{\Lambda}_s $ for $ s< \alpha$, $s\neq 1$,
\begin{eqnarray*}
   && \widetilde{\Lambda}_s(y)= \int \left(\mathcal{X}_y(x)-1\right)d \Lambda_s(x), \quad \mathrm{if} \quad 0<s<1,\\
   &&  \widetilde{\Lambda}_s(y)= \int\left(\mathcal{X}_y(x)-1-i\ipro{x}{y}\right) d \Lambda_s(x), \quad \mathrm{if} \quad 1<s<2,
\end{eqnarray*}
Then $\widetilde{\Lambda}_s$ depends  continuously on $(s,y)$ in $[0,\alpha] \times V\backslash \{0\}$ and  $\widetilde{\Lambda}_s$ satisfies:
\begin{eqnarray*}
  && \overline{P}_* \widetilde{\Lambda}_s(x) = \int \widetilde{\Lambda}_s (g^*x) d\bar{\mu}(g) =\kappa(s)\widetilde{\Lambda}_s(x), \mbox{ and }    \widetilde{\Lambda}_s (tx)= t^s   \widetilde{\Lambda}_s (x), \mbox{ for } t>0.
\end{eqnarray*}
  For $s<\alpha$, we define
  $$ \widehat{C}_s(v) = \int (\widetilde{\Lambda}_s(y+v) -\widetilde{\Lambda}_s(y) ) d\eta_v(y) $$
and we observe that by dominated convergence,
$$\lim_{s\rightarrow \alpha_-} \widehat{C}_s(v)= \int (\widetilde{\Lambda}(y+v)-\widetilde{\Lambda}(y) ) d \eta_v(y). $$
Hence $ \lim_{s\rightarrow \alpha_-} \widehat{C}_s(v)=\widehat{C}_\alpha(v)  $ if $ \alpha \neq 1$, while   $ \lim_{s\rightarrow \alpha_-} \widehat{C}_s(v)=\widehat{C}_\alpha(v) - i \gamma(v) $ if $ \alpha =1$.
On the other hand, $Z^*_0v=\sum_{k=0}^{\infty}M_0^*\cdots M_k^* v$ satisfies
 $Z_0^* v= M_0^* (Z^*v+ v)$, where $$Z^*=\sum_{k=1}^{\infty} M_1^*\cdots M_k^*$$ and $M_0^*$ is a copy of $M^*$ independent of $Z$.
 It follows:
 \begin{eqnarray*}
   \mathbb{E}(\widetilde{\Lambda}_s (Z^*_0v)) &=& \mathbb{E} \left[\int \widetilde{\Lambda}_s (g^*(Z^*v+v)) d\bar{\mu} (g)\right]  \\
    &=& \kappa(s) \mathbb{E}(\widetilde{\Lambda}_s (Z^*v+v) ),
 \end{eqnarray*}
 hence $$\widehat{C}_s(v)= \mathbb{E} (\widetilde{\Lambda}_s(Z^*v+v) ) -\mathbb{E}(\widetilde{\Lambda}_s (Z^*v) ) = \left(\frac{1}{\kappa(s)}-1\right) \mathbb{E}(\widetilde{\Lambda}_s (Z^*v) ).$$
 By Proposition \ref{rwgpro2-1}, the function  $\log\kappa(s)$  is  convex, hence $\kappa(s)$ has a left derivative $\kappa'(\alpha_-)$ at $s=\alpha$:
  $$m_\alpha=\lim_{s\rightarrow \alpha_{-} } \frac{1-\kappa(s)}{\alpha-s}.$$
    In order to get the value of $\widehat{C}_\alpha(v)$, we need to evaluate $ \lim_{s\rightarrow \alpha_-} (\alpha-s) \mathbb{E}(\widetilde{\Lambda}_s (Z^*v ) ).$

 For this purpose we will use Theorem \ref{rwgTh2-3}, we  write
 $$F_{s,v}(t)= \int_{\norm{x}\geq t} \widetilde{\Lambda}_s(\bar{x}) d\eta_v(x)$$
  and we  observe that $|F_{s,v}(t)|\leq \displaystyle\sup_{\bar{x} \in \mathbb{S}^{d-1}} |\widetilde{\Lambda}_s(\bar{x})|$ is bounded  by $K<+\infty$ on $[0,\alpha]$ by definition of $\widetilde{\Lambda}_s$. Also for $t\geq 0$:
 $$t^{\alpha} F_{s,v}(t) = \int_{\norm{x}\geq 1} \widetilde{\Lambda}_s(\bar{x} ) d\eta_v^t(x)$$
 with $\eta_v^t=t^{\alpha} (t^{-1}. \eta_v)$. Hence, using the convergence of $\eta_v^t$ to $\Delta_v$ for $t\rightarrow +\infty$ given by Theorem \ref{rwgTh2-3} and the fact that $\widetilde{\Lambda}^1$ is bounded   by $K<+\infty$ with $\Delta_v$-negligible discontinuities, we get  for $t$ large,
 $$t^{\alpha } F_{s,v}(t)= \Delta_v (\widetilde{\Lambda}_s^1 )+ c_s(t),$$
where $\widetilde{\Lambda}_s^1(x)= \widetilde{\Lambda}_s (\bar{x}) \mathbf{1}_{[1,\infty)} (\norm{x})$ and  $ c_s(t)=o(1) \mbox{ as } t\rightarrow +\infty$  uniformly in $s\in [0,\alpha]$. We note that uniformity of $o(1)$ is valid since the function $\widetilde{\Lambda}_s(\bar{x})$ is continuous and bounded on $[0,\alpha]\times \mathbb{S}^{d-1}$, hence $\widetilde{\Lambda}_s^1(x)$ is bounded by the $\Delta_v$-integrable function $K\mathbf{1}_{[1,\infty)} (\norm{x})$.
By definition of $F_{s,v}$:
\begin{eqnarray*}
   \mathbb{E}(\widetilde{\Lambda}_s (Z^*v))&=&\int \norm{y}^s \widetilde{\Lambda}_s(\bar{y}) d\eta_v(y)= \int_V \left(\int_{0<t\leq\norm{y}} st^{s-1} dt \right) \widetilde{\Lambda}_s(\bar{y}) d\eta_v(y)\\
   &=&\int_0^\infty sF_{s,v} (t) t^{s-1} dt .
\end{eqnarray*}
Let $\rho$ be a positive increasing function on $[0,\alpha)$ such that
\begin{equation*}
  \lim_{s\rightarrow \alpha_-} \rho(s)=+\infty,\quad  \lim_{s\rightarrow \alpha_-} (\alpha-s)\rho^s(s) =0,\quad  \lim_{s\rightarrow \alpha_-} \rho^{s-\alpha}(s)=1.
\end{equation*}
 One can take for example $\rho(s)=(\alpha-s)^{-\frac{1}{2\alpha}}.$
 Then to compute the required limit, we decompose the integral of $F_{s,v}(t)$ according to the function $\rho(s)$ and use the asymptotic
 expansions of $F_{s,v}(t)$:
 \begin{eqnarray*}
  \lefteqn{(\alpha - s)\mathbb{E}(\widetilde{\Lambda}_s(Z^* v)) = (\alpha-s) \int_0^{\rho(s)} sF_{s,v}(t)t^{s-1} dt} \\
  &&   { } + (\alpha - s) \int_{\rho(s)}^{\infty}s\Delta_v(\widetilde{\Lambda}^1_s) t^{-\alpha+s-1} dt  + (\alpha-s) \int_{\rho(s)}^{\infty} c_s(t) t^{-\alpha+s-1} d t .
 \end{eqnarray*}

 Notice that the limits of the first and third terms are zero. Indeed, by the properties of $\rho(s)$:
 \begin{equation*}
 \lim_{s\rightarrow \alpha_-} \left|(\alpha-s)\int_{0}^{\rho(s)} sF_{s,v}(t)t^{s-1} dt\right| \leq \lim_{s\rightarrow \alpha_-}(\alpha-s)\rho^s(s)\sup_{t>0} |F_{s,v}(t)|=0.
\end{equation*}
To compute the limit of the third term, let $\epsilon>0$ and observe that, using the above remark, there exists $s_0=s_0(\epsilon)<\alpha$ close to $\alpha$ such that
$|c_s(t)|<\epsilon$ for $t>\rho(s_0)$, hence  using again the properties of $\rho(s)$:
$$\lim_{s\rightarrow\alpha_-} \left| (\alpha-s)\int_{\rho(s)}^\infty c_s(t) t^{-\alpha +s-1} dt \right| \leq  \epsilon \lim_{s\rightarrow\alpha_-}\rho^{s-\alpha}(s)=\epsilon.$$
Since $\epsilon$ was arbitrary, we obtain that the limit above is in fact zero.
As a result, using again the properties of $\rho(s)$,
\begin{eqnarray*}
   \lim_{s\rightarrow \alpha_-} \widehat{C}_s(v)&=& \lim_{s\rightarrow \alpha_-} \left( \frac{1}{\kappa(s)}-1\right)\mathbb{E} (\widetilde{\Lambda}_s(Z^*v))  \\
   &=&m_\alpha\lim_{s\rightarrow \alpha_-}(\alpha-s)\int_{\rho(s)}^\infty s\Delta_v(\widetilde{\Lambda}^1_s) t^{-\alpha+s-1} dt   \\
   &=& m_\alpha \lim_{s\rightarrow  \alpha
   _-}s \Delta_v(\widetilde{\Lambda}^1_s)\lim_{s\rightarrow\alpha_-}\rho^{s-\alpha}(s)=\alpha m_{\alpha} \Delta_v (\widetilde{\Lambda}^1),
\end{eqnarray*}
since,  as above, $ \lim_{s\rightarrow \alpha_-}\widetilde{\Lambda}_s^1 =\widetilde{\Lambda}^1$ and $\widetilde{\Lambda}^1_s$ is uniformly bounded by a $\Delta_v$-integrable function.
 The statement  follows.
 \end{proof}

\section{Spectral gap properties of Fourier operators, eigenfunctions and eigenvalues}
We follow closely the method of \cite{GuivarchLePage08spectral,BuraczewskiGuivarch10PTRF} and we recall the corresponding functional space notations.

On continuous  functions on $V$ we introduce the semi-norm
$$[f]_{\varepsilon,\lambda}=\sup_{x\neq y} \frac{|f(x)-f(y)|}{\norm{x-y}^\varepsilon (1+\norm{x})^{\lambda} (1+\norm{y})^{\lambda}}$$
and the two norms
$$|f|_{\theta}=\sup_{x} \frac{|f(x)|}{(1+\norm{x})^{\theta}},\quad  ||f||_{\theta,\varepsilon,\lambda}=|f|_\theta+[f]_{\varepsilon,\lambda}.$$
Notice that the conditions $\lambda+\varepsilon\leq \theta $ (always assumed) and $[f]_{\varepsilon,\lambda} <\infty$ imply  $|f|_\theta <\infty$.
Define the Banach spaces $$ \mathbb{C}_{\theta}= \{f: |f|_{\theta}<\infty\}, \quad \mathbb{B}_{\theta,\varepsilon,\lambda} =\{ f: ||f||_{\theta,\varepsilon,\lambda} <\infty\}$$
and on them we consider the action of the transition operator $P$:
$$Pf(x)= \mathbb{E} (f(Mx+Q)) =\int f(hx) d\mu(h)$$
where $(Q,M)$ is a random variable distributed according to $\mu$. We consider also the Fourier operator $P_v$ defined by
$$P_vf(x)= P(\mathcal{X}_v f)(x) = \mathbb{E} [\mathcal{X}_v (Mx+Q) f(Mx+Q)]$$
where $v\in V$. Notice that $P_0=P$. We will prove later (Theorem \ref{rwgTh3-4}) that the operators $P_v$ are bounded on $\mathbb{B}_{\theta,\varepsilon,\lambda}$ for appropriately chosen parameters $\theta, \varepsilon,\lambda$.  Also, for $v$ small,
 they have a unique dominant eigenvalue $k(v)$ with $|k(v)|<1$ if $v\neq 0$, $k(0)=1$ and the rest of the spectrum of $P_v$ is contained in a disk
 of center $0$ and radius less than $|k(v)|$. For an operator $A$ we denote by $\sigma (A)$ its spectrum and by $r(A)$ its spectral radius.
These properties are based on the estimations below and  \cite{Ionescu50AM,KellerLiverani99spectrum}.
The following simple but basic fact was observed in \cite{GuivarchHardy88AIHP}. For reader's convenience, we give its proof.
\begin{pro}\label{rwgpro3-1} We have
$$P_v^n f(x) =\mathbb{E} ( \mathcal{X}_v(S_n^x) f(X_n^x)).$$
\end{pro}
\begin{proof} If $n=1$, then the formula above coincide with definition of $P_v$. By induction, we have
\begin{align*}
    P_v^n f(x)&=P(\mathcal{X}_v P_v^{n-1} f)(x) = \mathbb{E} [\mathcal{X}_v (Mx+Q) ( P_v^{n-1} f)(Mx+Q)]\\
   & =  \mathbb{E} [\mathcal{X}_v (Mx+Q) \mathcal{X}_v (S_{n-1}^{Mx+Q})f(X_{n-1}^{Mx+Q})]\\
    &=\mathbb{E} [\mathcal{X}_v (S_n^x ) f(X_n^x) ].
\end{align*}
\end{proof}

The following proposition  gives the basic estimations which allow  the use of \cite{Ionescu50AM}. Similar estimations were used in \cite{LePage83,LePage89} for  different purposes.
\begin{pro}\label{rwgpro3-2}
There exists $D=D(\theta)<\infty$ such that for any $v\in V$, $n\in\mathbb{N}$, $\theta<\alpha $ we have
\begin{equation}\label{rwg3-1}
|P_v^nf|_\theta \leq D|f|_{\theta}.
\end{equation}
If $2\lambda +\varepsilon <\alpha$, $\varepsilon<1$, $\theta <2\lambda$, there exist constants $C_1, C_2\geq 0, \rho \in [0,1)$ depending on $\theta,\varepsilon,\lambda$  such that for any $n \in \mathbb{N} $, $f\in \mathbb{B}_{\theta,\varepsilon,\lambda}$, $v\in V$,
\begin{equation}\label{rwg3-2}
[P_v^n f ]_{\varepsilon,\lambda} \leq C_1 \rho^n[f]_{\varepsilon,\lambda}+ C_2 \norm{v}^{\varepsilon}|f|_{\theta}.
\end{equation}
\end{pro}
\begin{proof}
Notice that
\begin{equation}\label{rwg3-3}
    X_n^x=X_n^y+ \Pi_n(x-y),
\end{equation}
where $\Pi_n= M_n M_{n-1}\cdots M_1$. Writing $X_n=X_n^0$, by Proposition \ref{rwgpro3-1}  we have
\begin{align*}
    |P_v^nf(x)|_{\theta}&\leq  \mathbb{E} \left[ \frac{|f(X_n^x)|}{(1+\norm{X_n^x})^{\theta}} \cdot \frac{(1+\norm{X_n^x})^{\theta}}{(1+\norm{x})^{\theta}}\right]
     \\ &\leq |f|_\theta \mathbb{E} \left[  \frac{(1+\norm{X_n} +\norm{\Pi_n x})^{\theta}}{(1+\norm{x})^{\theta}}\right]\\ &
     \leq 3^\theta |f|_\theta   \mathbb{E} ({1+\norm{X_n}^\theta + \mnorm{\Pi_n }^\theta} )
\\  &\leq     3^\theta |f|_\theta   \left(1+ \mathbb{E} \norm{X_n}^\theta + C (\kappa (\theta)+ \epsilon')^n \right)
\end{align*}
where $0<\epsilon'< 1-\kappa(\theta)$ and $C$ is a constant. If we set  $D=3^\theta \left(1+ \sup_n\mathbb{E} \norm{X_n}^\theta + C\right)<\infty$,  the first inequality \eqref{rwg3-1} follows.

Now we turn to the proof of \eqref{rwg3-2}. By Proposition \ref{rwgpro3-1}, we have
\begin{equation*}
    P_v^nf(x)-P_v^nf(y)= \mathbb{E} \Big[\mathcal{X}_v(S_n^x) \big(f(X_n^x)- f(X_n^y) \big)\Big]+\mathbb{E} \Big[\big(\mathcal{X}_v(S_n^x) -\mathcal{X}_v(S_n^y) \big) f(X_n^y)\Big].
\end{equation*}
Without loss of generality, assume that $\norm{x}\geq \norm{y}$. Let
\begin{align*}
 \textsc{J}_1(x,y)& = \frac{\left|\mathbb{E} \Big[\mathcal{X}_v(S_n^x) \big(f(X_n^x)- f(X_n^y) \big)\Big]\right|}{\norm{x-y}^{\varepsilon}(1+\norm{x})^{\lambda}(1+\norm{y})^{\lambda}}, \\
 \textsc{J}_2 (x,y)&=  \frac{\left|\mathbb{E}  \Big[\big(\mathcal{X}_v(S_n^x) -\mathcal{X}_v(S_n^y) \big) f(X_n^y)\Big]\right|}{\norm{x-y}^{\varepsilon}(1+\norm{x})^{\lambda}(1+\norm{y})^{\lambda}}.
\end{align*}
The first step is to estimate $\textsc{J}_1(x,y) $.
\begin{align*}
   \textsc{J}_1(x,y) & \leq \mathbb{E} \Big(|f(X_n^x)- f(X_n^y) |/( \norm{x-y}^{\varepsilon}(1+\norm{x})^{\lambda}(1+\norm{y})^{\lambda}) \Big)\\
   & \leq [f]_{\varepsilon,\lambda} \mathbb{E} \left( \frac{ \norm{X_n^x-X_n^y}^\varepsilon (1+ \norm{X_n^x})^{\lambda} (1+\norm{X_n^y})^{\lambda}}{ \norm{x-y}^{\varepsilon}(1+\norm{x})^{\lambda}(1+\norm{y})^{\lambda} }\right)\\
   & \leq [f]_{\varepsilon,\lambda} \mathbb{E}\left( \frac{\mnorm{\Pi_n}^\varepsilon (1+\norm{X_n} +\norm{\Pi_n x})^{\lambda}(1+\norm{X_n} +\norm{\Pi_n y})^{\lambda}}{(1+\norm{x})^{\lambda}(1+\norm{y})^{\lambda}} \right)\\
    &\leq [f]_{\varepsilon,\lambda} \mathbb{E} \left(\mnorm{\Pi_n}^\varepsilon (1+\norm{X_n}+ \mnorm{\Pi_n})^{2\lambda} \right)\\
    & \leq 3^{2\lambda}[f]_{\varepsilon,\lambda} \left( \mathbb{E} \mnorm{\Pi_n}^\varepsilon  + \mathbb{E} \mnorm{\Pi_n}^{2\lambda+\varepsilon}+   \left(\mathbb{E} \mnorm{\Pi_n}^{2\lambda+\varepsilon}\right)^{\frac{\varepsilon}{2\lambda+\varepsilon}} \left(\mathbb{E} \norm{X_n} ^{2\lambda+\varepsilon}\right)^{\frac{2\lambda}{2\lambda+\varepsilon} }  \right).
\end{align*}
Proposition \ref{rwgpro2-1} allows us to  choose $\epsilon_1>0$ and a constant $A_1$ such that
\begin{equation*}
 \max\{\kappa(\varepsilon), \kappa(2\lambda+\varepsilon)\} + \epsilon_1 <1,
\end{equation*}
and  for all $n\in \mathbb{N}$,
\begin{align*}
      \mathbb{E} \mnorm{\Pi_n}^{2\lambda+\varepsilon} \leq A_1 ( \kappa(2\lambda+\varepsilon)+ \epsilon_1)^n, \quad   \mathbb{E} \mnorm{\Pi_n}^{\varepsilon} \leq A_1 ( \kappa(\varepsilon)+ \epsilon_1)^n.
\end{align*}
Now setting
\begin{align*}
    \rho=\max\left\{ \kappa(\varepsilon)+ \epsilon_1, \kappa(2\lambda+\varepsilon)+\epsilon_1, (\kappa(2\lambda+\varepsilon)+\epsilon_1)^{\frac{2\lambda}{2\lambda+\varepsilon}}\right\}
\end{align*}
and
\begin{align*}
     C_1= 3^{2\lambda}\left(2A_1 +(A_1)^{\frac{\varepsilon}{2\lambda+\varepsilon}} \sup_n\left(\mathbb{E} \norm{X_n} ^{2\lambda+\varepsilon}\right)^{\frac{2\lambda}{2\lambda+\varepsilon}} \right),
\end{align*}
  we have
\begin{equation}\label{rwg3-4}
    \textsc{J}_1(x,y)\leq C_1\rho^n [f]_{\varepsilon,\lambda}.
\end{equation}
Now we are going to estimate $ \textsc{J}_2(x,y) $. Observe that
\begin{equation*}
    |e^{i\ipro{x}{y}}-1|\leq 2\norm{x}^{\varepsilon}\norm{y}^\varepsilon \quad \mathrm{and} \quad S_n^x-S_n^y= Z_n (x-y),
\end{equation*}
where $Z_n=\sum_{k=1}^n M_k\cdots M_1$. Using these facts, we get
\begin{align*}
    \textsc{J}_2(x,y)&\leq  2\norm{v}^\varepsilon |f|_\theta \mathbb{E}\left[ \frac{\norm{Z_n}^{\varepsilon} (1+\norm{X_n^y})^{\theta}}{(1+\norm{x})^{\lambda}(1+\norm{y})^{\lambda}}\right]\\
    &\leq 2\norm{v}^\varepsilon |f|_\theta \mathbb{E}\left[ \frac{\norm{Z_n}^{\varepsilon} (1+\norm{X_n}+\mnorm{\Pi_n}\cdot \norm{y})^{\theta}}{(1+\norm{x})^{\lambda}(1+\norm{y})^{\lambda}}\right]\\
    &\leq 2\cdot 3^{\theta}\norm{v}^\varepsilon |f|_\theta \mathbb{E}[\norm{Z_n}^{\varepsilon} (1+\norm{X_n}^{\theta}+\mnorm{\Pi_n}^{\theta})].
\end{align*}
To finish our proof, the left thing is to prove the uniform boundedness of the expectation in the last expression.
For $s<\alpha$, by the properties of $\kappa(s)$, there exists $\epsilon_s>0$ and a constant $A_s>1$ such that
\begin{equation*}
    \kappa(s)+\epsilon_s<1 \quad \mathrm{and} \quad \mathbb{E}\mnorm{\Pi_n}^{s} \leq A_s (\kappa(s)+\epsilon_s)^n.
\end{equation*}
Then if  $s<\min\{1, \alpha\}$,
\begin{equation*}
     \mathbb{E}\norm{Z_n}^{s} \leq 1+ \sum_{m=1}^{n}\mathbb{E}\mnorm{\Pi_m}^{s}\leq  A_s\sum_{m=0}^{n}(\kappa(s)+\epsilon_s)^m,
\end{equation*}
and if  $s\in [1, \alpha)$,
 \begin{equation*}
     \mathbb{E}\norm{Z_n}^{s} \leq \left(1+ \sum_{m=1}^{n}(\mathbb{E}\mnorm{\Pi_m}^{s})^{\frac{1}{s}}\right)^{s}\leq  \left(A_s\sum_{m=0}^{n}(\kappa(s)+\epsilon_s)^\frac{m}{s} \right)^{s} .
\end{equation*}
Therefore for $s<\alpha$,
\begin{equation*}
    \sup_n  \mathbb{E}\norm{Z_n}^{s} <\infty.
\end{equation*}
Also we have that $\sup_n \mathbb{E}\norm{X_n}^q <\infty$ for $q< \alpha$. Now noticing that $\theta+\varepsilon<\alpha$ and applying the H\"{o}lder inequality, we obtain that
\begin{equation*}
    \sup_n \mathbb{E}[\norm{Z_n}^{\varepsilon} (1+\norm{X_n}^{\theta}+\mnorm{\Pi_n}^{\theta})]<\infty.
\end{equation*}
We set $C_2= 2\cdot 3^{\theta}  \sup_n \mathbb{E}[\norm{Z_n}^{\varepsilon} (1+\norm{X_n}^{\theta}+\mnorm{\Pi_n}^{\theta})] $ and thus
\begin{equation}\label{rwg3-5}
   \textsc{J}_2(x,y)\leq C_2 \norm{v}^{\varepsilon}|f|_{\theta}.
\end{equation}
Finally combining  \eqref{rwg3-4} and \eqref{rwg3-5}  we obtain that
\begin{equation*}
    [P_v^n f ]_{\varepsilon,\lambda} \leq \sup_{x,y} (\textsc{J}_1(x,y)+\textsc{J}_2(x,y) ) \leq C_1 \rho^n[f]_{\varepsilon,\lambda}+ C_2 \norm{v}^{\varepsilon}|f|_{\theta}.
\end{equation*}
\end{proof}

\begin{pro}\label{rwgpro3-3}
Assume that $2\lambda+\epsilon<\alpha$. Then, for any $v\neq 0$, the equation $P_v f= z f$, $|z|=1$, $f\in \mathbb{B}_{\theta,\varepsilon,\lambda}$ implies $f=0$. In particular, $r(P_v)<1$.
\end{pro}
If $ \texttt{supp}\bar{\mu}$ consists of similarities, this is Lemma 3.14 in \cite{BuraczewskiGuivarch10PTRF} ; in view of its role here we give the proof.
\begin{proof}[Proof of Proposition  \ref{rwgpro3-3}]
Assume that $P_v f= zf$ for some nonzero $f\in \mathbb{B}_{\theta,\varepsilon,\lambda}$. Then the function $f$ is bounded. Indeed for every $n$
\begin{equation*}
    |f(x)|=|z^n f(x)|\leq P^n(|f|)(x),
\end{equation*}
hence
\begin{equation*}
    |f(x)| \leq \lim_{n\rightarrow \infty}P^n (|f|)(x)= \eta(|f|).
\end{equation*}
Next observe that since $f$ is continuous, on the support of $\eta$ the function $|f|$ is equal to its maximum and without  loss of generality we may assume that this  maximum is 1.
For every $n$ and $x\in \texttt{supp} \eta$,  noticing that $z^nf(x)= \mathbb{E}[e^{i\ipro{v}{S_n^x}} f(X_n^x)]$ and using a convexity argument, we can show  that
\begin{equation*}
    z^n f(x)= e^{i\ipro{v}{S_n^x}} f(X_n^x) \quad \mathbb{P}\mbox{-a.e.}
\end{equation*}
Hence for every $x, y \in \texttt{supp} \eta$,
\begin{equation}\label{rwg3-6}
     \frac{f(x)}{f(y)}e^{i\ipro{v}{Z_n(y-x)}}=\frac{f(X_n^x)}{f(X_n^y)},
\end{equation}
where $Z_n=\sum_{k=1}^n M_k\cdots M_1$.
By the H\"{o}lder inequality and since $|f|=1$, we have
\begin{align*}
    &~~~\limsup_{n\rightarrow \infty} \mathbb{E}\left| \frac{f(X_n^x)}{f(X_n^y)}-1\right|\\ &\leq [f]_{\varepsilon,\lambda}  \limsup_{n\rightarrow \infty} \mathbb{E}[\norm{X_n^x-X_n^y}^\varepsilon (1+\norm{X_n^x})^\lambda (1+\norm{X_n^y})^\lambda]\\
    & =[f]_{\varepsilon,\lambda} \limsup_{n\rightarrow \infty} \mathbb{E}[\norm{M_n\cdots M_1 (x-y)}^\varepsilon (1+\norm{X_n^x})^\lambda (1+\norm{X_n^y})^\lambda ]\\
    &\leq  [f]_{\varepsilon,\lambda} \norm{x-y}^\varepsilon \limsup_{n\rightarrow \infty}\left[\mathbb{E}\mnorm{M_n\cdots M_1}^{2\lambda+\varepsilon}\right]^{\frac{\varepsilon}{2\lambda+\varepsilon}}\\
    &~~~ \cdot \limsup_{n\rightarrow \infty}\left[\mathbb{E} (1+\norm{X_n^x})^{\lambda+\frac{\varepsilon}{2}} (1+\norm{X_n^y})^{\lambda+\frac{\varepsilon}{2}} \right]^{\frac{2\lambda}{2\lambda+\varepsilon}}.
\end{align*}
By our assumption, the first limit is zero and the second one is finite. Hence
\begin{equation*}
   \limsup_{n\rightarrow \infty} \mathbb{E}\left| \frac{f(X_n^x)}{f(X_n^y)}-1\right| =0.
\end{equation*}
Therefore for $\mathbb{P}$ a.e. trajectory $ \omega$  there exists a sequence $\{n_k\}=\{n_k(\omega)\}$ such that
\begin{equation*}
     \lim_{n_k\rightarrow \infty}  \frac{f(X_{n_k}^x)}{f(X_{n_k}^y)}=1.
\end{equation*}
By Proposition \ref{rwgpro2-5}, $\lim_{n\rightarrow \infty} Z_n(\omega)=Z(\omega)$  exists a.s.. Hence letting $k\rightarrow \infty$ we obtain that there is $\Omega_0$ such that $\mathbb{P}(\Omega_0)=1$ and for $\omega\in \Omega_0$,
\begin{equation*}
     \frac{f(x)}{f(y)}=e^{i\ipro{v}{Z(\omega)( x-y)}}=e^{i\ipro{Z^*(\omega)v}{ x-y}}.
\end{equation*}
We are going to prove that this leads to a contradiction whenever $v\neq0$. We choose $x_j,y_j \in \texttt{supp}\eta$, $j=1,\cdots, d$ with $x_j-y_j$ spanning $V$ as a vector space. Such points exist because the support of $\eta$, as a set invariant under the action of $\texttt{supp} \mu$, is not contained in some proper affine subspace of $V$. Let $\eta_v$ be the law of $W(\omega)=Z^*(\omega)v$. Then for every $j$  the support of $\eta_v$ is contained in the union of affine hyperplanes $\bigcup_{n\in \mathbb{Z}}\{H_j+ ns_jv_j\}$, where $H_j$  is some hyperplane orthogonal to $v_j=x_j-y_j$ and $s_j$ is appropriately chosen constants. Taking intersection of all such sets defined for every $j$ we conclude that
$\texttt{supp}\eta_v$ is contained in some discrete set of points, hence $\texttt{supp}\eta_v$ is discrete. This contradicts Proposition \ref{rwgpro2-2}.

For the last assertion we observe that in view of Theorem of Ionescu Tulcea and Marinescu  \cite{Ionescu50AM}, if $z$ belongs to the spectrum of $P_v$ and $|z|=1$ then $z$ is an eigenvalue of $P_v$. \end{proof}

The following theorem corresponds to  items 1-3 of Theorem \ref{rwgTh1-6} and is our basic tool for the study of $P_v$.
\begin{theorem}\label{rwgTh3-4}
Assume $\theta,\varepsilon,\lambda$ satisfy $0<\varepsilon<1,$ $2\lambda +\varepsilon<\alpha$, $\theta \leq 2\lambda$. Then $P_v$ has the following properties:
\\ 1) $P_v$ is a bounded operator on $\mathbb{B}_{\theta,\varepsilon,\lambda}$ with spectral radius $r(P_v) \leq 1$;
\\ 2) If  $v\neq 0$, $r(P_v)<1$;
\\ 3) If $v=0$, $P=P_0$ satisfies $ P\mathbf{1}=\mathbf{1}$ , $\eta P =\eta $. The operator $\Q$ on $\mathbb{B}_{\theta,\varepsilon,\lambda}$ defined by $\Q f =P f -\eta (f) \mathbf{1} $ has spectral radius less than 1 and $\eta (\Q f)= 0$. In other words, $P$ is the direct sum of  the Identity  on $\mathbb{C}\mathbf{1}$ and of an operator on  \texttt{Ker} $\eta$ with spectral radius strictly less than 1.
\end{theorem}
\begin{proof}[Proof of Theorem \ref{rwgTh3-4}]

Proposition \ref{rwgpro3-2}   implies that $P_v$ is a power-bounded operator on $\mathbb{B}_{\theta,\varepsilon,\lambda} $,    
hence assertion 1 follows. Since bounded subsets of $(\mathbb{B}_{\theta,\varepsilon,\lambda}, || \cdot||_{\theta,\varepsilon,\lambda} )$ are relatively compact in $(\mathbb{C}_\theta, |\cdot|_{\theta})$, the inequality in part 2 of Proposition \ref{rwgpro3-2} shows that we can apply the theorem of Ionescu-Tulcea  and Marinescu (see \cite{Ionescu50AM}) to $ P_v$. In particular, if for some $v\in V$, $r(P_v)=1$, there exists $f\in \mathbb{B}_{\theta,\varepsilon,\lambda}$ and $z\in \mathbb{C}, |z|=1$, $f\neq 0$ such that $P_v f=z f$. If $v\neq 0$, this contradicts Proposition \ref{rwgpro3-3}, hence assertion 2 follows.

If $v=0$, part 2 of Proposition \ref{rwgpro3-2} gives:
$ [P^{n_0} f]_{\varepsilon,\lambda} \leq \rho_1 [f]_{\varepsilon,\lambda}$
for some $n_0 \in \mathbb{N}, \rho_1\in [0,1[$. We show that $f\rightarrow [f]_{\varepsilon,\lambda}$ defines a norm  equivalent to $ f\rightarrow |f|_{\theta}$ on the subspace  $\texttt{Ker} \eta= \{f\in \mathbb{B}_{\theta,\varepsilon,\lambda} ; \eta(f)=0 \} $.
 Since $\eta(f)=0$, if $f\in \texttt{Ker} \eta$, the condition $[f]_{\varepsilon,\lambda}=0$ implies $f=0$. Hence $f\rightarrow [f]_{\varepsilon,\lambda}$ is a norm on $\texttt{Ker} \eta $,  which satisfies $[f]_{\varepsilon,\lambda} \leq ||f||_{\theta,\varepsilon,\lambda}$.
 Since $ \varepsilon\leq 1$ we have
 \begin{align*}
    |f(x)-f(y)| &\leq [f]_{\varepsilon,\lambda} \norm{x-y}^{\varepsilon} (1+\norm{x}^{\lambda})(1+\norm{y}^{\lambda})\\ &\leq 4 [f]_{\varepsilon,\lambda}(1+ \norm{x}^{\lambda+\varepsilon})(1+\norm{y}^{\lambda+\varepsilon}).
 \end{align*}
 Since $\lambda+\varepsilon<\theta<\alpha$, we have $1+\norm{x}^{\lambda+\varepsilon} \leq 2(1+\norm{x}^{\theta})$ and $\int \norm{y}^{\lambda+\varepsilon} d\eta (y)=D<\infty$.  Hence, using $\eta(f)=0$:
 $$|f(x)|\leq 8(1+D) [f]_{\varepsilon,\lambda} (1+\norm{x}^{\theta}),$$
 i.e. $|f|_\theta \leq 8 (1+D) [f]_{\varepsilon,\lambda}$. The equivalence of norms follows .

 We can write $\mathbb{B}_{\theta,\varepsilon,\lambda} = \mathbb{C}\mathbf{1} \oplus \texttt{Ker} \eta$.
 Since $P\mathbf{1}=\mathbf{1}$ and $\eta P=\eta$, the subspaces $\mathbb{C}\mathbf{1}$ and $\texttt{Ker} \eta$ are closed $P$-invariant subspaces of $\mathbb{B}_{\theta,\varepsilon,\lambda}$. Since $\Q\mathbf{1}=0$, $\Q$ can be identified with its restriction to $\texttt{Ker} \eta$. Then the inequality
 $[\Q^{n_0} f]_{\varepsilon,\lambda} \leq \rho_1 [f]_{\varepsilon,\lambda}$ and the equivalence of norms observed above imply
 \begin{equation*}
    r(\Q^{n_0}) \leq \rho_1, \quad r(\Q) \leq \rho_1^{1/n_0}<1.
 \end{equation*}
   \end{proof}

 The study of $P_{tv}$ for $t$ small and $v$ fixed is based on a theorem of Keller and Liverani(\cite{KellerLiverani99spectrum}), Proposition \ref{rwgpro3-2} and the following
easy lemma.
\begin{lemma}\label{rwglem3-5}
If $\lambda+2\varepsilon <\theta <\alpha$, $\delta\leq \varepsilon$, there exists $C>0$ such that for any $\gamma \in [\lambda+2\varepsilon, \theta]$
and $v,w\in V$:
$$|(P_v-P_w)f|_{\gamma}\leq C\norm{v-w}^{\delta} ||f||_{\theta,\varepsilon, \lambda}.$$
\end{lemma}
\begin{proof}
We observe that
\begin{multline}
    |(P_v-P_w) f(x)|\leq \int |e^{i\ipro{v}{hx}} - e^{i\ipro{w}{hx}}| |f(hx)| d\mu(h)
\\ \leq 2 \norm{v-w}^{\delta} \int\norm{hx}^{\delta}|f(hx)-f(0)| d\mu(h) +  2 \norm{v-w}^{\delta} |f(0)| \int \norm{hx}^{\delta} d\mu(h) \\
 \leq    2 \norm{v-w}^{\delta}[f]_{\varepsilon,\lambda}\int\norm{hx}^{\delta+\varepsilon}(1+\norm{hx})^{\lambda} d\mu(h) + 2 \norm{v-w}^{\delta} |f|_{\theta} \int\norm{hx}^\delta d\mu(h).
\end{multline}

Therefore if we take $C= \sup_x \{ 2 \int [\norm{hx}^{\delta+\varepsilon}(1+\norm{hx} )^{\lambda} + \norm{hx}^{\delta} ] d\mu(h)/(1+\norm{x})^{\lambda+2\varepsilon}  \}$, then
\begin{equation*}
  |(P_v-P_w)f|_{\gamma}= \sup_x |(P_v-P_w)f(x)/ (1+\norm{x})^{\gamma}|\leq C\norm{v-w}^{\delta} ||f||_{\theta,\varepsilon, \lambda}.
\end{equation*}
\end{proof}

 In view of Theorem  \ref{rwgTh3-4} and Lemma \ref{rwglem3-5}, we may use the perturbation theorem of \cite{KellerLiverani99spectrum} for the family $P_{tv}$, hence as in \cite{GuivarchLePage08spectral,BuraczewskiGuivarch10PTRF} we have the following
 \begin{pro}\label{rwgpro3-6}
 Assume $\varepsilon<1$, $\lambda+2\varepsilon <\theta \leq 2\lambda <2\lambda +\varepsilon <\alpha$, $v\in V$.
 Then there exists $t_0>0$, $\delta>0$, $\rho<1-\delta$ such that for every $t\in \mathbb{R}$ with $|t|\leq t_0$:
 \\ a) The spectrum of $P_{tv}$ acting on $\mathbb{B}_{\theta,\varepsilon,\lambda}$  is contained  in $ \mathfrak{ S}=\{ z\in \mathbb{C} ; |z|\leq \rho \} \bigcup \{ z\in \mathbb{C} ; |z-1|<\delta\}$.
 \\ b) The set $\sigma (P_{tv} ) \bigcap \{z\in \mathbb{C}; |z-1|\leq \delta \}$ consists of exactly  one eigenvalue $k(tv)$, the corresponding eigenspace is one dimensional and $\lim_{t\rightarrow 0} k(tv)=1$.
 \\ c) If $\pi_{tv}$ is the spectral projection on the above eigenspace of $P_{tv}$,  there exists an operator $\Q_{tv}$ with $r(\Q_{tv}) \leq \rho$, $\pi_{tv} \Q_{tv}= \Q_{tv}\pi_{tv}=0$ and for every $n\in \mathbb{N}$, $f\in \mathbb{B}_{\theta,\varepsilon,\lambda}$,
 $$ P_{tv}^n f = k^n(tv)\pi_{tv}(f)+ \Q_{tv}^n(f) .$$
 Furthermore $k(tv)$, $\pi_{tv}$, $\Q_{tv}$ depends continuously on $t$.\\
d) For any $z$ in the complement of $\mathfrak{S}$:
$$||(z-P_{tv})^{-1}f||_{\theta, \varepsilon, \lambda}
\leq D ||f||_{\theta, \varepsilon, \lambda}$$
for some constant $D$ independent of $t$.
 \end{pro}
 This statement allows us to complete the proof of Theorem \ref{rwgTh1-6}.
For $t$ small define the function $g_{tv}= \pi_{tv} (\mathbf{1})$. Hence $$ P_{tv}g_{tv}=k(tv)g_{tv}.$$ Then for any function $f$ in $ \mathbb{B}_{\theta,\varepsilon,\lambda}$ we define $\mathcal{E}_{tv}(f) \in \mathbb{C}$ by $\pi_{tv}(f)= \mathcal{E}_{tv}(f) g_{tv}$.

We will be able to get the asymptotic expression of $k(tv)$ for $t$ small through the use of a new family of operators $T_{t,v}$
 on $\mathbb{B}_{\theta,\varepsilon,\lambda}$ defined  by
 $$T_{t,v}f(x)= \int \mathcal{X}_{tb}(x+v) f(g^*(x+v)) d\mu(h). $$
Then $T_v=T_{0,v}$, $T_v \eta_v=\eta_v$, where $\eta_v$ is the stationary measure for the Markov chain $W_n$. It turns out that
the analogues of Theorem \ref{rwgTh3-4}, Proposition \ref{rwgpro3-6}, are valid for the family $T_{t,v}$.
Therefore, for small values of $t$, the spectrum of $T_{t,v}$ in some neighborhood of 1 consists of only one point $k^*(t,v)$ which satisfies
$|k^*(t,v)|= r(T_{t,v})$.  We denote by $T^*_{t,v}$ the dual operator on $\mathbb{B}^*_{\theta,\varepsilon,\lambda}$ of $T_{t,v}$.
 One observes that for any $ v\in V$, the function $\mathcal{X}_v$ belongs to $\mathbb{B}_{\theta,\varepsilon,\lambda}$ and $|| \mathcal{X}_{v}||_{\theta,\varepsilon,\lambda}\leq 1+2\norm{v}^{\varepsilon}$. It follows that for any $ \mathcal{E} \in \mathbb{B}_{\theta,\varepsilon,\lambda}^*$, $$\widehat{\mathcal{E}} (v):=\mathcal{E}(\mathcal{X}_{v} )$$ plays the role of a Fourier transform for $ \mathcal{E}$ and $$|\widehat{\mathcal{E}}(v)|\leq (1+2\norm{v}^\varepsilon) ||{\mathcal{E}}||_{\theta,\varepsilon,\lambda}.$$
The following relation between $P_{tv}$ and $T_{t,v}$ plays an essential role in the calculation of the asymptotic expansion for $k(tv)$.
\begin{pro}\label{rwgpro3-8}
For any $t\in \mathbb{R}$, $v\in V\backslash \{0\}$, $\mathcal{E} \in \mathbb{B}^*_{\theta,\varepsilon,\lambda}$,
$$ P_{tv}(\widehat{\mathcal{E}}\circ t )= (\widehat{T_{t,v}^* \mathcal{E}}) \circ t.$$
\end{pro}
\begin{proof}
As in \cite{GuivarchLePage08spectral}, the proof is based on the definitions of $\mathcal{X}_{x}$, $T_{t,v}$ and the fact that the map $x\rightarrow tx$ commute with $x\rightarrow gx$ for $g\in G$. However, in view of its role here, we give it explicitly.
Since $x\rightarrow gx$ ($g\in G$) and $x\rightarrow tx$ commute:
\begin{eqnarray*}
 T_{t,v}(\mathcal{X}_{tx} )(y)  &=&  \int \mathcal{X}_{tb} (y+v) \mathcal{X}_{tx}(g^*(y+v)) d\mu(h)\\
  &=& \int \mathcal{X}_{tb} (y+v) \mathcal{X}_{t(gx)}(y+v) d\mu(h) =\int \mathcal{X}_{t(gx+b)}(y+v)d\mu(h); \\
   P_{tv}(\widehat{\mathcal{E}}\circ t) (x)&=& \iint\mathcal{X}_{tv}(gx+b) \mathcal{X}_y (t(gx+b))d\mathcal{E}(y) d\mu(h)  \\
   &=&\iint \mathcal{X}_{y+v} (t(gx+b))d \mathcal{E}(y) d\mu(h)=  \mathcal{E}(T_{t,v}(\mathcal{X}_{tx}))= \widehat{T_{t,v}^*\mathcal{E}} (tx).
\end{eqnarray*}
\end{proof}
As in \cite{GuivarchLePage08spectral}, this proposition  allows us to construct an eigenfunction of $P_{tv}$  from an eigenfunctional  $\eta_{t,v}$ of $T_{t,v}$, hence  in Section 4 it will lead to the expansion of $k(tv)$ at $t=0$, using the following result (see \cite[Corollary 2]{GuivarchLePage08spectral}):
\begin{cor}\label{rwgpro3-9}
Assume $\eta_{t,v} \in \mathbb{B}^*_{\theta,\varepsilon,\lambda}$ satisfies $$ T_{t,v}^*\eta_{t,v}=k^*(t,v)\eta_{t,v},\quad  \eta_{t,v}(\mathbf{1})={1}.$$
If $\varepsilon<1/2$, there exists $t_3>0$ such that if $|t|\leq t_3$, the function $$\psi_{tv}= \widehat{\eta_{t,v}}\circ t$$ is the unique  normalized eigenfunction of $P_{tv}$ (with value $1$ at $0$) acting on  $\mathbb{B}_{\theta,\varepsilon,\lambda}$ and corresponding to the eigenvalue $ k(tv)$, i.e.
$$P_{tv}(\psi_{tv})=k(tv)\psi_{tv}, \quad \psi_{tv}(0)=1 .$$
Moreover $k(tv)=k^*(t,v)$ and
\begin{equation*}
    (k(tv)-1)\eta( \psi_{tv})=\eta (\psi_{tv} (\mathcal{X}_{tv}-1)).
\end{equation*}
\end{cor}

\begin{remark}\label{rwgre3-9}In particular, using assertion c of Proposition \ref{rwgpro3-6}, we see that $ \lim_{t\rightarrow 0 } || \psi_{tv}-1||_{\theta,\varepsilon,\lambda}=0$. Since $\eta$ defines an element of $ \mathbb{B}_{\theta,\varepsilon,\lambda}$, we have $ \lim_{t\rightarrow 0} \eta(\psi_{t,v})=1$.

\end{remark}
\section{Asymptotic expansion of eigenvalues in terms of tails and the proof of Theorem \ref{rwgTh1-5} }

Using the techniques of \cite{KellerLiverani99spectrum,LePage89} and the above results, we deduce from Proposition \ref{rwgpro3-6} the following result (see \cite[Proposition 3.18]{BuraczewskiGuivarch10PTRF}):
\begin{pro}\label{rwgpro3-7}
Assume additionally that $\lambda+ 3\varepsilon<\theta $, $2\lambda+3 \varepsilon< \alpha $. Then the  identity embedding of $\mathbb{B}_{\theta,\varepsilon,\lambda}$ into $ \mathbb{B}_{\theta,\varepsilon, \lambda+\varepsilon}$ is continuous and the decomposition $P_{tv}=k(tv) \pi_{tv}+ \Q_{tv}$ coincide on both spaces. Moreover, there exist constants $D>0$ and $ t_1>0$ such that for $|t|\leq t_1$,
 we have if $\norm{v}\leq 1$:
 \\ (i) $ ||(P_{tv}-P)f||_{\theta,\varepsilon,\lambda+\varepsilon} \leq D|t|^{\varepsilon} ||f||_{\theta,\varepsilon,\lambda}$;
 \\ (ii) $ ||(k(tv)\pi_{tv}-\pi_0)f||_{\theta,\varepsilon,\lambda+\varepsilon}  \leq D|t|^{\varepsilon} ||f||_{\theta,\varepsilon,\lambda}$;
 \\ (iii) $ ||(\pi_{tv}-\pi_0)f||_{\theta,\varepsilon,\lambda+\varepsilon}\leq D|t|^{\varepsilon}||f||_{\theta,\varepsilon,\lambda}$ ;
\\ (iv) $ ||(\Q_{tv}-\Q) f||_{\theta,\varepsilon,\lambda} \leq D|t|^{\varepsilon}||f||_{\theta,\varepsilon,\lambda}$
\\ (v) $||g_{tv}-\mathbf{1}||_{\theta,\varepsilon,\lambda}\leq D|t|^\varepsilon$;
\\(vi) $ |k(tv)-1|\leq D|t|^\varepsilon$;
\\ (vii) $\mathcal{E}_{tv}$ is a bounded functional on $\mathbb{B}_{\theta,\varepsilon,\lambda}$ with norm at most $D|t|^{\varepsilon}$.
\end{pro}

The following theorem is a consequence of Propositions \ref{rwgpro3-9}, \ref{rwgpro3-7} and the homogeneity at infinity of stationary measures, given by Theorem \ref{rwgTh2-3}. It is a detailed form of Theorem \ref{rwgTh1-3}.
\begin{theorem}\label{rwgTh4-1}
a) If $0<\alpha<1$, then $$\lim_{t\rightarrow 0_+} \frac{k(tv)-1}{t^{\alpha}}=C_{\alpha}(v)$$ with
$$C_{\alpha}(v)=\int (\mathcal{X}_{v}(x)-1 ) \widehat{\eta_v}(x) d \Lambda (x).$$
b) If $\alpha =1 $, then
\begin{equation*}
    \lim_{t\rightarrow 0_+} \frac{k(tv)-1-i\ipro{v}{\delta(t)}}{t}=C_1(v)
\end{equation*}
with
\begin{equation*}
    C_1(v)=\int_V\left((\mathcal{X}_v(x)-1)\widehat{\eta_v}(x)-i\frac{\ipro{v}{x}}{1+\norm{x}^2}\right) d \Lambda (x)
\end{equation*}
and  $\delta(t)=\displaystyle \int_V \frac{tx}{1+\norm{tx}^2}d\eta(x)$. Furthermore there exists a constant $K_\star$ with
$K_\star=c_1+4A/\log 2 + \int_{\norm{x}>1} \norm{x}/(1+\norm{x}^2) d\Lambda(x)$
($A$ is given in Theorem \ref{rwgTh2-3})
 such that
 \begin{equation*}
|\delta(t)|\leq \left\{ \begin{array}{ll}
                            K_\star|t||\log t|,& |t|\leq 1/2,  \\
                            K_\star|t|, &   |t|>1/2.
                        \end{array}
\right.
 \end{equation*}
\\ c) If $1<\alpha<2$,
$$ \lim_{t\rightarrow 0_+}\frac{k(tv)-1-i\ipro{v}{tm}}{t^\alpha} =C_\alpha (v)$$
where $$C_\alpha (v)= \int \left( (\mathcal{X}_v(x)-1)\widehat{\eta_v}(x)- i\ipro{v}{x}  \right) d \Lambda (x).$$
d) If $\alpha=2$,
\begin{equation*}
    \lim_{t\rightarrow 0} \frac{k(tv)-1-i\ipro{v}{tm}}{t^2|\log|t||} =2C_2(v)
\end{equation*}
where $$C_2(v)=-\frac{1}{4} \int (\ipro{v}{w}^2 +2\ipro{v}{w}\eta_v(w^*)) d \sigma_2(w)$$ is a quadratic form.
\\ e) If $\alpha>2$, then
\begin{equation*}
    \lim_{t\rightarrow 0} \frac{k(tv)-1-i \ipro{v}{tm}}{t^2}= C_{2+}(v)
\end{equation*}
with $$C_{2+}(v)=-\frac{1}{2} q(v,v)- q(v, (I-z^*)^{-1} z^*v) .$$
\end{theorem}
The proof is based on estimations of $P_{tv} $, $\psi_{tv}$, $\eta$, which are valid here, as in \cite[Theorem 5.1]{BuraczewskiGuivarch10PTRF}; these estimations are formal consequences of the homogeneity statements in Theorem \ref{rwgTh2-3}, Corollary \ref{rwgpro3-9}, which in turn correspond to relations (2.2), (2.3)  and Lemma 3.23 of \cite{BuraczewskiGuivarch10PTRF}.

To prove our theorem \ref{rwgTh4-1}, we need further properties of the stationary measure $\eta$. In particular, essential use  is made of the homogeneity at infinity of $\eta$ stated in Theorem \ref{rwgTh2-3}.  Also Lemmas \ref{rwglem4-2}, \ref{rwgcor4-3}, \ref{rwglem4-4} are used in the proof. The comparisons stated in these lemmas are based on the general Lemma  \ref{rwglem4-1},  will allow to estimate expressions of the form $\int_V f(t,x) d\eta(x)$ for $|t|$ small. We denote by $I_1$ the interval $[-1,1]$.
 For the proof of the lemmas below, see \cite[section 4]{BuraczewskiGuivarch10PTRF}.
\begin{lemma}\label{rwglem4-1}
 Let $f$ be any continuous function on  $I_1\times V$ satisfying
  \begin{equation}\label{rwg4-1a}
    |f(t,x)|\leq  \left\{  \begin{array}{ll}
                    D_{\delta,\beta}|t|^{\delta+\beta}\norm{x}^{\beta}, &{for ~~~\norm{tx}>1 ;} \\
                     D_{\delta,\gamma}|t|^{\delta+\gamma}\norm{x}^{\gamma}, & {for ~~~\norm{tx}\leq 1,}
                    \end{array}
                  \right.
   \end{equation}
  where $\beta<\alpha$, $\gamma+\delta>\alpha$ and $\delta>0$. Then
  \begin{equation*}
    \lim_{t\rightarrow 0}\frac{1}{|t|^{\alpha}} \int_V f(t,x)d\eta(x)=0.
  \end{equation*}
  \end{lemma}
Now we present some properties of the eigenfunction $\psi_{tv}$. To do this, we will need some further hypotheses on the parameters $\theta, \varepsilon, \lambda$ and from now on, we will assume additionally that
\begin{eqnarray*}
   && \mathrm{ if }~~~ 1<\alpha<2,~~~ \mathrm{ then }~~~ 1+\lambda+\varepsilon>\alpha, \\
   && \mathrm{if }~~~  \alpha=2, ~~~\mathrm{ then }~~~ \lambda+2\varepsilon>1, \\
   &&   \mathrm{if }~~~  \alpha>2, ~~~   \mathrm{ then }~~~ \lambda=1.
\end{eqnarray*}
It is easy to prove that there exists  $\theta, \varepsilon, \lambda$ satisfying all the assumptions in our theorems and the conditions above.
 \begin{lemma}\label{rwglem4-2}
    There exists $D''$ such that \begin{eqnarray*}
                                   &&|\psi_{tv}(x)-\widehat{\eta}_v(tx)|\leq D''|t|^{2\varepsilon}\norm{x}^{\varepsilon}, \quad for~~ \norm{tx}>1;  \\
                                   &&  |\psi_{tv}(x)-\widehat{\eta}_v(tx)|\leq D''|t|^{\varepsilon}\norm{tx}^{\tau},\quad for~~ \norm{tx}\leq 1,
                                 \end{eqnarray*}
 for $\tau=\min\{1, \lambda+\varepsilon\}$. \end{lemma}
\begin{cor}\label{rwgcor4-3}
    If $\alpha \leq 2$, then
    \begin{equation*}
    \lim_{t\rightarrow 0} \frac{1}{|t|^\alpha}\int_V\left(\mathcal{X}_v(tx)-1\right)\left(\psi_{tv}(x)-\widehat{\eta}_v(tx)\right)d\eta(x)=0.
    \end{equation*}
\end{cor}
We will need also the speed of convergence of $\eta(\psi_{tv})$ to 1.
\begin{lemma}\label{rwglem4-4}
 Assume  $v$ is fixed. Then there exists $D'''>0$ and $t_3>0$ such that for $|t|<t_3$, we have
 \begin{equation*}
    |1-\eta(\psi_{tv})|\leq D'''|t|^{\min\{1,\lambda+\varepsilon\}}.
 \end{equation*}
 \end{lemma}
As an example of how to use the above estimations and the basic Theorem \ref{rwgTh2-3}, let us consider in more detail the cases $\alpha <1$ and $ \alpha=1$. For the cases $\alpha \in ]1,2]$, $\alpha>2$ we refer to \cite[section 5]{BuraczewskiGuivarch10PTRF}.
\begin{proof}[Proof of Theorem \ref{rwgTh4-1}]
\emph{Case} $\alpha<1$. We use the expression of $\psi_{tv}, k(tv)$ given by Corollary \ref{rwgpro3-9} and
 write for $t>0$,
      \begin{eqnarray*}
  \frac{1}{t^{\alpha} } (k(tv)-1)\eta(\psi_{tv})     &=&   \frac{1}{t^{\alpha} } \int (\mathcal{X}_v(tx)-1) \psi_{tv} (x)d\eta(x) \\
     &=& \frac{1}{t^{\alpha} } \int (\mathcal{X}_v(tx)-1)  \widehat{\eta_v} ( tx)d\eta(x)
      \\ & & + \frac{1}{t^{\alpha} } \int (\mathcal{X}_v(tx)-1) (\psi_{tv}(x)-\widehat{\eta_v} ( tx) )d\eta(x).
   \end{eqnarray*}
We observe that the function $f_v=(\mathcal{X}_v-1)\widehat{\eta_v}$ satisfies the regularity and growth conditions of Theorem \ref{rwgTh2-3} since $f_v(x)$ is bounded and $|f_v(x)|\leq 2\norm{x}$ for $\norm{x}\leq 1$. Hence the first term converges to
$$ \int (\mathcal{X}_v(x)- 1)\widehat{\eta_v}(x) d\Lambda (x).$$
The use of Corollary \ref{rwgcor4-3} shows that the second term has limit zero, hence the result follows from Remark \ref{rwgre3-9}.

\emph{Case} $\alpha=1$.
Using Corollary \ref{rwgpro3-9}, we see that
\begin{eqnarray*}
 \lefteqn{t^{-1}[k(tv)-1-i\ipro{v}{\delta(t)}]}   \\
   &=&[t\eta(\psi_{tv})]^{-1} \Bigl[ \Bigl( \eta(\psi_{tv}(\mathcal{X}_{tv}-1))-i\ipro{v}{\delta(t)} \Bigr) + i \bigl(1-\eta(\psi_{tv})\bigr)\ipro{v}{\delta(t)}  \Bigr]\\
   &=& [\eta(\psi_{tv})]^{-1} \big[\J_{11}(t)+ \J_{12}(t)+\J_{13}(t)\big],
\end{eqnarray*}
where
\begin{align*}
     \J_{11}(t)&= t^{-1}\int_V \Bigl(\widehat{\eta}_v(tx)(\mathcal{X}_{tv}(x)-1)-it\frac{\ipro{v}{x}}{1+\norm{tx}^2}\Bigr)d\eta(x), \\
     \J_{12}(t)&= t^{-1}\int_V(\psi_{tv}(x)-\widehat{\eta}_v(tx))(\mathcal{X}_{tv}(x)-1)d\eta(x),  \\
     \J_{13}(t)&=i t^{-1} \bigl(1-\eta(\psi_{tv})\bigr)\ipro{v}{\delta(t)}.
\end{align*}
By Corollary \ref{rwgcor4-3},
\begin{equation}\label{rwg4-2}
    \lim_{t\rightarrow 0_+}\J_{12}(t)=0.
\end{equation}

Next observe that the function $ f_1 (x)=\widehat{\eta}_v(x)(\mathcal{X}_{v}(x)-1)-i\frac{\ipro{v}{x}}{1+{\norm{x}^2}}$ satisfies the growth condition \eqref{rwg2-1} in Theorem \ref{rwgTh2-3}.
Indeed $f_1$ is bounded  and for $\norm{x}\leq 1$,
 \begin{eqnarray*}
     |f_1(x)|&=&|(\widehat{\eta}_v(x)-1)(\mathcal{X}_{v}(x)-1) |+|\mathcal{X}_{v}(x)-1-i\frac{\ipro{v}{x}}{1+\norm{x}^2}|   \\
     &\leq & 2\norm{v} \cdot \norm{x}\cdot  ||\mathcal{X}_{x}-1||_{\theta,\varepsilon,\lambda}+ 4(\norm{v}\cdot\norm{x})^2 \\
&\leq& 8\norm{v}\cdot \norm{x}^{1+\lambda+\varepsilon}+4(\norm{v} \cdot \norm{x})^2,
 \end{eqnarray*}
where in the last step, we  use the estimation
$$||\mathcal{X}_{x}-1||_{\theta,\varepsilon,\lambda}\leq 4\norm{x}^{\min\{1,\lambda+\varepsilon\}},$$
which can be shown by direct calculation. Thus by Theorem \ref{rwgTh2-3}, we have that
\begin{equation}\label{rwg4-3}
\lim_{t\rightarrow 0_+} \J_{11}(t)= \Lambda(f_1)=C_1(v).
\end{equation}
Now the left thing is to evaluate the term $\J_{13}(t)$.

We first need to show the following properties of $\delta(t)$:
\begin{equation}\label{rwg4-4}
    |\delta(t)|\leq \left\{
                      \begin{array}{ll}
                        K_\star|t| , & \quad ~~~|t|\geq \frac{1}{2} ;  \\
                        K_\star|t\log|t|| , &\quad ~~~|t|< \frac{1}{2},
                      \end{array}
                    \right.
\end{equation}
with $K_\star=c_1+4A/\log 2 + \int_{\norm{x}>1} \norm{x}/(1+\norm{x}^2) d\Lambda(x)$, $c_1$ a constant and $A$ given by Theorem \ref{rwgTh2-3}.
For $|t|\geq 1/2$, \eqref{rwg4-4} is obvious.

For $|t|<1/2$, we write
  \begin{align*}
     |\delta(t)|&\leq \int_{V} \norm{tx}/(1+\norm{tx}^2)d\eta(x)\\
       & = \int_{\norm{x}\leq 1}\frac{\norm{tx}}{1+\norm{tx}^2}d\eta(x)+\int_{1<\norm{x}\leq \frac{1}{|t|}}\frac{\norm{tx}}{1+\norm{tx}^2}d\eta(x)+\int_{ \norm{x}> \frac{1}{|t|}}\frac{\norm{tx}}{1+\norm{tx}^2}d\eta(x) .
  \end{align*}
  The first integral is bounded by $|t|$. By Theorem \ref{rwgTh2-3}, the third one, divided by $|t|$, converges to $\int_{\norm{x}>1} \frac{\norm{x}}{1+\norm{x}^2} d\Lambda(x)$ as $|t|$ tends to 0. Applying Theorem \ref{rwgTh2-3}, we see that
  \begin{align*}
    \int_{1<\norm{x}\leq \frac{1}{|t|}} \frac{\norm{tx}}{1+\norm{tx}^2} d\eta(x)&\leq |t| \sum_{k=0}^{|\log_2|t||} 2^{k+1} \eta(\norm{x}\geq 2^k)\\
    & \leq A|t|  \sum_{k=0}^{|\log_2|t||} 2^{k+1} 2^{-k} \leq \frac{4}{\log 2}A|t||\log|t||.
  \end{align*}
  (Here by convention, when $|\log_2|t||$ is not an integer, the summands are for  all $k$  no larger than $|\log_2|t||$ ).
   Then \eqref{rwg4-4} follows.
 Combining \eqref{rwg4-4} with Lemma \ref{rwglem4-4}  we   obtain
 \begin{equation}\label{rwg4-5}
 \lim_{t\rightarrow 0_+}\J_{13}(t)=0.
 \end{equation}
 By relations \eqref{rwg4-2},\eqref{rwg4-3} and \eqref{rwg4-5},  we have
 \begin{equation*}
  \lim_{t\rightarrow 0_+} \frac{k(tv)-1-i\ipro{v}{\delta(t)}}{|t|}=C_1(v).
\end{equation*}
\end{proof}

\begin{proof}[Proof of Theorem \ref{rwgTh1-5}]
In view of the continuity theorem, it is enough to justify that the characteristic functions of the normalized sums $S_n^x$ converge pointwise to a function which is continuous at zero  and to show full non degeneracy of the corresponding law.  The convergence follows easily from the asymptotic
 expansion of $k(tv)$ at $t=0$ given by Theorem \ref{rwgTh4-1}.
 Also if $\alpha \in [0,2]$, using formula \eqref{rwg1-5} for $ C_{\alpha}(v)$,  the non degeneracy proof is based on $ Re C_{\alpha}(v) < 0$ for $v\neq 0$ and is  the same as  in \cite{BuraczewskiGuivarch10PTRF}, since, using Theorem \ref{rwgTh2-3}, $\texttt{supp}\Lambda$ is not contained in a hyperplane and $\Delta_v\neq 0$ is $\alpha$-homogeneous. If $\alpha>2$, the argument is the same as in \cite{BuraczewskiGuivarch10PTRF} and is based on the order 2 differentiability of $k(t)$, since for $t\neq 0$
 $r(P_{tv})<1$, which follows from Theorem \ref{rwgTh3-4}. The invertibility of $I-z^*$ follows from the fact that $r(z^*)=r(z) <1$, which is itself a consequence of $r(z) =\lim_{n\rightarrow \infty} (\mathbb{E}(\mnorm{M}^n))^{1/n} \leq \lim_{n\rightarrow \infty} (\mathbb{E} (\mnorm{M_n\cdots M_1}))^{1/n}=\kappa(1)<1 $.

\end{proof}

\section{On the limit laws of the normalized Birkhoff sums}

Here we use the results of \cite{GuivarchLePage10} in order to give more precise formulas for $ C_\alpha(v)$  defined by \eqref{rwg1-5}. For a Radon measure $\rho$ we denote by $\breve{\rho}$ the push-forward of $\rho$ by the symmetry $x\rightarrow -x$.
 We recall from Section 2 that the $\rho_{\alpha}(\overline{\mu})$-stationary probability measure $\sigma_{\alpha} $ on $\mathbb{S}^{d-1}$ was defined by $\Lambda =c \sigma_\alpha \otimes \ell^\alpha$ with $c>0$. In order to write detailed  formulas  for $\Delta_v$ $(v\in V\backslash \{0\})$ we need to distinguish two cases I and II.  In case I,
 $[\texttt{supp}\overline{\mu} ]$ and $[\texttt{supp}\overline{\mu}]^*$ have no invariant convex cone and  we have $ \Delta_v= c^*(v) \sigma_\alpha^* \otimes \ell^\alpha$ where $ c^*(v)  >0$ if $v\neq 0$ and $  \sigma_\alpha^* $ is the unique $\rho_{\alpha}(\overline{\mu}^*)$-stationary probability measure on $\mathbb{S}^{d-1}$.  In case II, there are two extremal   $\rho_{\alpha}(\overline{\mu}^*)$-stationary  measures on $\mathbb{S}^{d-1}$, $\sigma_\alpha'$ and $ \sigma_\alpha''$, which are symmetric of each other ( hence $ \sigma_\alpha''=\breve{\sigma}_\alpha'$ ) and which are supported by the two  $[\texttt{supp}\overline{\mu}]^*$-minimal subsets of $\mathbb{S}^{d-1}$. Then, using Theorem C of \cite{GuivarchLePage10}, we get that there exists two nonnegative functions $c'(v)$, $c''(v)$ such that
 $$ \Delta_v= c'(v) ( \sigma'_\alpha \otimes \ell^\alpha ) + c''(v)( \sigma''_\alpha \otimes \ell^\alpha ) $$
  and $c^*(v)=c'(v)+c''(v) >0$ for $v\neq 0$.

\begin{pro}\label{rwgpro5-1}
With the above notations we have, if $\alpha \in (0,1)\cup(1,2):$

In case I,  $ \Delta_v( \widetilde{\Lambda}^1) = r_\alpha c^*(v) $,
 where $ r_\alpha= ( \sigma^*_\alpha \otimes \ell^\alpha) ( \widetilde{\Lambda}^1) <0$, $c^*(v)>0 $ if $ v\neq 0$,  $ c^*(v) $ is $ \alpha$-homogeneous, and  $ c^*(-v) =  c^*(v)$. In particular the stable limit law for $S_n^x$ is symmetric.

In case II, $ \Delta_v( \widetilde{\Lambda}^1) = c'(v) \gamma_\alpha + c'(-v) \overline{\gamma}_\alpha   $, where $ \gamma_\alpha= ( \sigma'_\alpha \otimes \ell^\alpha) ( \widetilde{\Lambda}^1)  $, $ Re \gamma_\alpha < 0 $,  $c^*(v)=c'(v)+c'(-v) >0$ if  $v\neq 0$, and $ c'(v)$ is $\alpha$-homogeneous.

\end{pro}

  \begin{proof}
  In view of the above observations, it remains to study  $ c^*(v)$,  $\sigma_\alpha^* $, $ c'(v)$,  $c''(v)$. This follows from Proposition \ref{rwgpro2-5}, in particular from the relations
  \begin{equation*}
     \Delta_{tv}= t^{\alpha} \Delta_{v} \mbox{ for }  t>0  \mbox{ and }\Delta_{-v}= \breve{\Delta}_v .
  \end{equation*}

 In case I, using $ \Delta_v= c^*(v) ( \sigma_\alpha^* \otimes \ell^\alpha)$  and the symmetry of $ \Delta_v$, $ \Delta_{-v}$, we get that $  \sigma_\alpha^*$ is symmetric and $ c^*(v)=c^*(-v)$.
 The symmetry of $  \sigma_\alpha^*$  gives that  $ r_\alpha= ( \sigma^*_\alpha \otimes \ell^\alpha) ( \widetilde{\Lambda}^1) $ is real and the condition $ Re C_\alpha(v)< 0$ gives $ r_\alpha <0$.

 In case II, the symmetry of $ \Delta_v$, $ \Delta_{-v}$ gives:
 \begin{equation*}
    c'(-v)( \sigma_\alpha' \otimes \ell^\alpha)+ c''(-v) ( \sigma_\alpha''\otimes \ell^\alpha)
    =  c'(v)( \breve{\sigma}_\alpha' \otimes \ell^\alpha)+ c''(v) ( \breve{\sigma}_\alpha''\otimes \ell^\alpha)
 \end{equation*}
 Since $\sigma_\alpha'$ and $\sigma_\alpha''=\breve{\sigma}_\alpha'$ are supported by  disjoint sets, we have
 $c'(-v)= c''(v)$. Also since $ \gamma_\alpha= ( \sigma'_\alpha \otimes \ell^\alpha) ( \widetilde{\Lambda}^1)  $,  we have   $ ( \sigma_\alpha''\otimes \ell^\alpha)(\widetilde{\Lambda}^1 )= ( \breve{\sigma}_\alpha'\otimes \ell^\alpha)(\widetilde{\Lambda}^1 )  = \overline{\gamma}_\alpha$.

 The homogeneity of $ c^*(v)$, $c'(v)$  follows from the relation $ \Delta_{tv}=t^\alpha \Delta_v$ if $t>0$.
  \end{proof}

  Few informations on the constant $c$, which enters  in the expression of $C_{\alpha } (v)$, seem to be available in the literature for $d>1$. See
\cite{Goldie91AAP,GuivarchLePage08spectral} for $d=1$. Furthermore, in order to deal with estimation problems in extreme value analysis of generalized GARCH models (see \cite{Mikosch2000}),
  we need to have control on the  function $C_\alpha(v)$. To go further, we  use the results of \cite{GuivarchLePage10}; hence we complete the notations already introduced. For $s\in [0,s_\infty)$ we denote by $\nu^*_s $ the unique probability on $\mathbb{P}^{d-1}$ which satisfies
\begin{equation*}
  (\nu^*_s\otimes \ell^s)\overline{P}_*= \kappa(s)  (\nu^*_s\otimes \ell^s)
\end{equation*}
and we write $p(s)= \int \norm{\ipro{\bar{x}}{\bar{y}} }^s d \nu_s(\bar{x}) d\nu^*_s(\bar{y})$. We consider the function $ e_s$ on $\mathbb{P}^{d-1}$ (or $\mathbb{S}^{d-1}$) given by
\begin{equation*}
 p(s)e_s(\bar{x})= \int \norm{\ipro{\bar{x}}{\bar{y}} }^s  d\nu^*_s(\bar{y}),
\end{equation*}
so that $\nu_s(e_s)=1$.

We know from \cite[Theorem 2.6]{GuivarchLePage10} that $e_s$ is continuous, positive and that the function $f_s$ on $V$ defined by $f_s(v)=e_s(\bar{v})|v|^s $ satisfies $ \bar{P} f_s= \kappa (s) f_s$.

In case II, there exist two probability measures $\theta_s$ (resp. $\theta^*_s$) on $\mathbb{S}^{d-1}$, which are symmetric to each other and are extremal solutions of the equation
\begin{equation*}
  (\theta_s\otimes \ell^s )\bar{P} = \kappa(s) (\theta_s\otimes \ell^s) \quad \Big({\mbox{resp. } (\theta^*_s\otimes \ell^s )\bar{P}_* = \kappa(s) (\theta^*_s\otimes \ell^s)}\Big).
\end{equation*}
 We denote these solutions by $\sigma_{s,+}, \sigma_{s,-}$(resp.  $\sigma'_{s}, \sigma''_s $). We define $c_+, c_-$ by $c\sigma_\alpha= c_+ \sigma_{\alpha,+} + c_-\sigma_{\alpha,-}$.

 Define the function $e_{s,+}$ on $\mathbb{S}^{d-1}$ by
 \begin{equation*}
    p(s)e_{s,+}(\bar{x}) = \int  \ipro{\bar{x}}{\bar{y}}_{+}^s  d\sigma'_s(\bar{y})
 \end{equation*}
 where $\ipro{\bar{x}}{\bar{y}}_{+}= \sup(\ipro{\bar{x}}{\bar{y}},0 )$. So that $\sigma_{s,+}(e_{s,+})=1$, and the function $f_{s,+}$ on $V$ given by $ f_{s,+}(v) = e_{s,+}(\bar{v})|v|^{s}$ satisfies $ \bar{P} f_{s,+}= \kappa(s) f_{s,+}$.

 We will use the quantities $d=c_+-c_-, d^*(v)= c'(v)-c'(-v)$.  For $\theta\geq 0 $ we will also consider the Banach space $\mathbb{C}_{\theta}$ already introduced in Section 3, and the weak topology on its dual space, a space which consists of the finite measures on $V$  with finite moment of order $\theta$.
   This topology will be called weak topology of order $\theta$. With the notations of Section 2, we consider the law $\eta'$ of the random variable $R-Q$, where
\begin{equation*}
  R=Q+ \sum_{k=1}^{\infty} M_1M_2\cdots M_k Q_{k+1}.
\end{equation*}
 This measure $\eta'$ plays an important role in the discussion of $ C_{\alpha}(v)$, due to the following proposition, first part of which extends previous results of (\cite{Goldie91AAP,BuraczewskiGuivarch09PTRF}).

\begin{pro}\label{rwgpro5-2}
 With the above notations, we have for $0<s<\alpha:$
 \begin{equation*}
   (\eta-\eta')(f_s) =(1-\kappa(s)) \eta(f_s), \quad (\mbox{ resp. } (\eta-\eta')(f_{s,+}) =(1-\kappa(s)) \eta(f_{s,+})  ).
 \end{equation*}
The function $(\alpha -s) \eta(f_s)$ $(resp. (\alpha -s) \eta(f_{s,+}) )$ extends analytically to $[0,\alpha+\delta]$ and
\begin{eqnarray*}
   && \lim_{s\rightarrow \alpha-} (\alpha -s) \eta(f_s)= m_{\alpha}^{-1}(\eta-\eta')(f_\alpha)=c  \\
   (resp. &&   \lim_{s\rightarrow \alpha-} (\alpha -s) \eta(f_{s,+})=m_{\alpha}^{-1} (\eta-\eta')(f_{\alpha,+})=c_+).
\end{eqnarray*}
In particular, if $ \tau$ is a probability  on $V$ and $\mu=\tau \otimes \bar{\mu}$ on $ H= V\rtimes G$, then $c ~(resp.~ c_+)$ depends continuously on $\tau $ in the weak topology of order $\alpha$.
\end{pro}
\begin{proof}
We denote $ \widehat{s}=\max (s,1)$, $K_1= \sup \{ e_s(\bar{v}); s\in [0,\alpha+\delta], \bar{v} \in \mathbb{S}^{d-1}\}$ and take  $ \epsilon\in (0,1) $ such that $0<s\leq \alpha -\epsilon$. Since $ f_{s}(v)= e_s(\bar{v})|v|^s$, Proposition \ref{rwgpro2-2} gives that $ f_s(R)$ is dominated on $[0, \alpha-\epsilon]$ by $K_1 \mathbb{E} (1+| R|^{\alpha-\epsilon}) <+\infty $. Hence $\eta (f_s)= \mathbb{E} (f_s(R)) $  defines a continuous function on $[0,\alpha)$. The same argument is valid for $ \eta'(f_s)=\mathbb{E} (f_s(R-Q))$. With the notations of Section 2, we have $R= Q+MR^1$,  where $R^1$ is independent of $(Q,M)$ and has the same law as $R$. It follows  $\eta-\eta' = \eta-\eta \bar{P}$, hence on $[0,\alpha)$,
\begin{equation*}
  (\eta-\eta')(f_s)= \eta(f_s)- \eta(\bar{P}f_s) =(1-\kappa(s))\eta(f_s),
\end{equation*}
\begin{equation*}
(\alpha-s)\eta(f_s)= \frac{\alpha-s}{1-\kappa(s)} (\eta-\eta')(f_s).
\end{equation*}
 Using the strict convexity of $\kappa(s)$ given by  Proposition \ref{rwgpro2-1}, we know that
\begin{equation*}
  \lim_{s\rightarrow \alpha-} \frac{\alpha-s}{1-\kappa(s)}= \frac{1}{m_\alpha} \mbox{ and  } 1-\kappa(s)\neq 0  \mbox{ for  } s\neq 0,\alpha,
\end{equation*}
hence $  \frac{\alpha-s}{1-\kappa(s)}$ defines a continuous function on $[0, \alpha+\delta]$. On the other hand, we have $ (\eta -\eta') (f_s)= \mathbb{E} (f_s(R) -f_s(R-Q))$.   But from above,  we have $ p(s) f_s(v) = \int \norm{\ipro{v}{\bar{y}}}^s d \nu_s^*(y)$, hence
\begin{equation*}
   p(s) |f_s(v)-f_s(v')| \leq \widehat{s} |v-v'|^s.
\end{equation*}
\begin{equation*}
   p(s) |f_s(R)-f_s(R-Q)| \leq \widehat{s} |Q|^s \leq \max(\alpha+\delta,1) (1+ |Q|^{\alpha+\delta}).
\end{equation*}
Hence $ f_s(R) -f_s(R-Q) $  is dominated on $ [0, \alpha+\delta]$ by the $ \mathbb{P}$-integrable function  $\max(1,\alpha+\delta) (1+ |Q|^{\alpha+\delta}) $.  It follows that $ (\eta-\eta')(f_s)$ is well defined as an analytical function  on $[0,\alpha +\delta]$ and gives the required  extension of $ (\alpha-s) \eta(f_s)$ to $[0, \alpha +\delta]$.

Next we are going to prove the formula
\begin{equation*}
\lim\limits_{s\rightarrow \alpha-} (\alpha-s) \eta(f_s) =c.
\end{equation*}
The proof is similar to the calculation of the limit of $\widehat{C}_s(v)$ in the proof of Proposition \ref{rwgpro2-5}, hence we give only a sketch.

Denote  $ H_s(t) = \int e_s(\bar{v})\mathbf{1}_{[t,\infty)}(|v|) d\eta(v)$ for $s\in [0,\alpha], t\in (0,\infty)$.
Then we will get that
     \begin{eqnarray*}
       \eta(f_s)  &=& \int e_s(\bar{v}) |v|^s d \eta(v) = \int_V  \bigg( \int_0^{|v|} st^{s-1} dt \bigg)e_s(\bar{v}) d \eta(v)  \\
        &=& \int_0^\infty st^{s-1} H_s(t) dt.
     \end{eqnarray*}
Observe that  for $ s\in [0,\alpha], t\in (0,\infty)$, $|H_s(t)|\leq K_1 <\infty. $
If we denote $\eta^t= t^{\alpha}( t^{-1}.\eta)$, then
\begin{equation*}
 t^{\alpha} H_{s}(t)= \int \mathbf{1}_{[1,\infty)} (|v|)e_s(\bar{v}) d \eta^t(v).
\end{equation*}
By Theorem \ref{rwgTh2-3}, the equicontinuity of the family $e_s(\bar{v})$,  and the fact that $ e_s(\bar{v})$ is bounded by $K_1$ with $\Lambda$-negligible discontinuities, we get for $t$ large,
    \begin{equation*}
      t^{\alpha} H_s(t)= c\sigma_\alpha(e_s) \ell^{\alpha}(1,\infty) + \epsilon_s(t)=c\alpha^{-1} \sigma_\alpha(e_s)  + \epsilon_s(t),
    \end{equation*}
where $\epsilon_s(t)=o(1)$ as $t\rightarrow\infty$ uniformly in $s\in [0,\alpha]$.
As before, we take a function $\rho(s)$on $ [0,\alpha)$, which satisfies
 \begin{equation*}
 \lim_{s\rightarrow \alpha_-} \rho(s)=+\infty,\quad  \lim_{s\rightarrow \alpha_-} (\alpha-s)\rho^s(s) =0,\quad  \lim_{s\rightarrow \alpha_-} \rho^{s-\alpha}(s)=1.
 \end{equation*}
 Now we decompose the integral $(\alpha-s) \eta(f_s)$:
  \begin{eqnarray*}
    \lefteqn{(\alpha-s) \eta(f_s)  =  (\alpha-s) \int_0^{\rho(s)} sH_s(t)t^{s-1}dt } \\
     &&{ } + (\alpha - s) \int_{\rho(s)}^{\infty}s  t^{-\alpha+s-1} c\alpha^{-1} \sigma_\alpha(e_s) dt  + (\alpha-s) \int_{\rho(s)}^{\infty} c_s(t) t^{-\alpha+s-1} d t .
  \end{eqnarray*}
 The first term and the third term tend to zero, and the second term tends to $c$. So we get that $ \lim\limits_{s\rightarrow \alpha-} (\alpha-s) \eta(f_s)=c $.


The same proof gives the corresponding formula  $ \lim\limits_{s\rightarrow \alpha-} (\alpha -s)\eta(f_{s,+}) =c_+$.

 In order to show the last assertion,  we use the formula  $ m_\alpha c =(\eta-\eta') (f_\alpha)$ and we observe that $ (\eta-\eta') (f_{\alpha} ) =\eta(\bar{\tau}) $ with $ \bar{\tau}(v) =   \int (f_{\alpha}(v)-f_{\alpha}(v- q) ) d { \tau} (q) $.
 We note the following  four properties of $\bar{\tau},\eta$:
 \begin{eqnarray*}
     && |\bar{\tau}(v)|\leq K_1\int |q|^{\alpha} d\tau(q),  \\
     && |\bar{\tau}(v)-\bar{\tau}(v') | \leq 2K_1 \max(\alpha,1) |v-v'|^{\alpha},  \\
     && \int |v|^{\epsilon} d\eta(v) \leq C\int |q|^{\epsilon} d\tau(q) \mbox{~ with ~} C= \mathbb{E} \Big( 1+\sum_{k=1}^{\infty} |M_1\cdots M_{k}|^{\epsilon} \Big)<\infty.\\
 &&\mbox{ If }\lim_{n\rightarrow \infty} \tau_n=\tau, \mbox{ then } \lim_{n\rightarrow \infty} \eta_n=\eta.
 \end{eqnarray*}
 In the last property the limits are taken in weak topology and $\eta_n$  is the stationary measure corresponding to $\tau_n$.

 The continuity of $c$ depending on $\tau$ follows since if $\tau_n$ converges to $\tau$ in the weak topology of order $\alpha$, then if
 $\bar{\tau}_n(v)= \int (f_\alpha(v)-f_{\alpha}(v-q)) d\tau_n (q) $,   the first two properties above imply the dominated convergence of $\bar{\tau}_n$ to $\bar{\tau}$, and the last one gives
the convergence of $\eta_n(\bar{\tau}_n)$ to $\eta(\bar{\tau})$.

 The proofs of the first two formulae are based on the definition of $\bar{\tau}$.  The third formula follows from Proposition \ref{rwgpro2-2}.

 For the last property, we know that,  because of the third property, the sequence $\eta_n$ is relatively compact in the weak topology. If $ P_n$ is the convolution operator on $V$ corresponding to $ \mu_n= \tau_n \otimes \bar{\mu} $, and if the subsequence $\eta_{n_k} $ converges weakly to $\eta^1$, then $\eta_{n_k} P_{n_k}$ converges weakly to $ \eta^1 P$. Hence $\eta^1P=\eta^1 $,  $\eta^1=\eta$ and $\eta_n$  converges weakly to $\eta$. The analogous result for $c_+$ follows from a corresponding argument.
 \end{proof}

 The formula $C_\alpha (v) =\alpha m_{\alpha} \Delta_{v} (\widetilde{\Lambda}^1 )$ can be made more explicit as follows
 \begin{pro}
 With the above notations, for $\alpha \in (0,2)$, $\alpha\neq 1$, we have:

 In case I: $C_\alpha (v) =-m_\alpha \frac{\Gamma (1-\alpha)}{\alpha } p(\alpha) cc^*(v) \cos \frac{\alpha \pi}{2}$

 In case II:  $C_\alpha (v) =-m_\alpha \frac{\Gamma (1-\alpha)}{\alpha } p(\alpha) \Big(c c^* (v)\cos \frac{\alpha \pi}{2}- idd^*(v)\sin \frac{\alpha \pi}{2} \Big)$.
       In particular, $C_\alpha (v)$ is real if and only if $c_+=c_-$.
 \end{pro}
 \begin{proof}
 We use the classical formula(see \cite{IbraginovLinnik71Indpendent}):
 \begin{equation*}
    \int_0^{\infty} \frac{e^{itx}-1}{x^{\alpha+1}} d x = - \frac{\Gamma(1-\alpha)}{\alpha} |t|^{\alpha} e^{-i\frac{\alpha \pi}{2}} \mbox{ if }  0<\alpha<1, t>0.
 \end{equation*}
If $t<0$, the value of the corresponding integral is the complex conjugate of the above integral;  for $1<\alpha <2$, the same result is valid for the integral $\int_0^{\infty} \frac{e^{itx}-1-itx}{x^{\alpha+1}} d x   $ instead of the left hand side of the formula.

In case I, the definition of $\Delta_v(\widetilde{\Lambda}^1) $ gives for $ 0<\alpha <1$:
\begin{equation*}
  \Delta_v(\widetilde{\Lambda}^1)= c c^*(v)\int  ( \mathcal{X}_{\bar{y}}  (t\bar{x}) -1 ) \mathbf{1}_{(0,\infty)}(t) \mathbf{1}_{[1,\infty)}(t') d \sigma_{\alpha}(\bar{x}) d \sigma_{\alpha}^*(\bar{y}) d \ell^{\alpha} (t)   d \ell^{\alpha} (t')
\end{equation*}
We note that  $ \alpha \ell^{\alpha} (1,\infty)=1 $ and by the symmetry property of $\sigma_\alpha, \sigma_\alpha^*$:
\begin{equation*}
  p(\alpha) =2 \int  \ipro{\bar{x}}{\bar{y}}_+^{\alpha} d \sigma_{\alpha}(\bar{x}) d \sigma^*_{\alpha}(\bar{y}).
\end{equation*}
Then we get:
\begin{eqnarray*}
 \alpha \Delta_v(\widetilde{\Lambda}^1) &=& cc^*(v)\int  d \sigma_\alpha(\bar{x}) d \sigma^*_{\alpha} (\bar{y}) \int_0^{\infty} \frac{e^{it \ipro{\bar{x}}{\bar{y}} } -1 }{t^{\alpha+1}}  dt, \\
  &=&    -\frac{\Gamma(1-\alpha)}{\alpha} cc^*(v) \Big(  e^{i\frac{\alpha\pi}{2}} +e^{-\frac{\alpha\pi}{2}} \Big) \int \ipro{\bar{x}}{\bar{y}}_+^\alpha d \sigma_\alpha(\bar{x}) d \sigma^*_{\alpha} (\bar{y})    \\
 &=&    -\frac{\Gamma(1-\alpha)}{\alpha} cc^*(v) p(\alpha) \cos \frac{\alpha\pi}{2}.
\end{eqnarray*}
The stated formula follows and remains valid for $1<\alpha<2$.

In case II, the calculation is similar, using the definitions of $c,d,c^*(v), d^*(v)$.

In case I, $C_\alpha(v)=\alpha m_{\alpha} \Delta_v(\widetilde{\Lambda}^1)$ is real, as the above formula shows.

In case II, the formula gives that $C_{\alpha}(v)$ is real if and only if $dd^*(v)=0$ for any $v\neq 0$, i.e. $(c_+-c_-) (c'(v)-c'(-v))=0$.

If $v\in \texttt{supp}(\sigma'_{\alpha})$, the convex cone generated by $\texttt{supp}(\sigma'_{\alpha})$ is invariant under $ \texttt{supp}(\mu_v^*)$.  It follows that the measures $\eta_v, \Delta_v$ are supported by this cone, hence $c'(v)>0$, $c'(-v)=0$, and $c'(v)-c'(-v)>0$. Then the condition $ (c_+-c_-) (c'(v)-c'(-v))=0$ for any $v\neq 0$ is equivalent to $c_+=c_-$.
 \end{proof}
For $ t>0$, we consider the automorphism $ u_t$ of $H$ defined by $ u_t(h) =(tb, g)$ where $h=(b,g)$, and we write $u_t(\mu) $ for the push-forward of $\mu$ by $u_t$.

If $ \mu$ satisfies condition $C$  and $\eta$ is the corresponding stationary measure we denote:
\begin{equation*}
  \Lambda(\mu) =\lim_{x\rightarrow 0_+ } x^{-\alpha} (x.\eta),  \quad \Lambda(t) = \Lambda (u_t(\mu)),
\end{equation*}
 and we write $c(t), c_+(t), c_-(t), C_{\alpha}(v,t)$ for the quantities $ c,c_+,c_-,C_{\alpha}(v)$ associated with $ u_t(\mu)$. Furthermore,
 let $\tau_0$ be a probability on $V$ such that $\int |q|^{\alpha+\delta} d \tau_0(q)<\infty,$
 \begin{equation*}
  \tau_t= (1-t)\tau_0+ t\breve{\tau}_0, \quad \mu_t = \tau_t \otimes \bar{\mu} \quad  (t\in [0,1])
 \end{equation*}
and denote also by  $c_+^t,c_-^t,c^t,C_\alpha^t(v)$  the quantities $c_+,c_-,c,C_\alpha(v)$ associated with $\mu_t$.  We see that $\mu_t $ satisfies Condition  $C$, since $\bar{\mu}$ satisfies condition $i$-$p$ and $d>1$.  In the following corollary we gather some consequence of the above propositions,  which give information on the above quantities.
\begin{cor}
  For $t\in \mathbb{R}^*$, we have $\Lambda (t)= t.\Lambda$. If $t>0$, then $c(t)= t^\alpha c, c_+(t)=t^{\alpha} c_+$, $ C_{\alpha} (v,t)= t^{\alpha} C_\alpha(v)$.

  If the law of $Q$ is symmetric, then $c_+=c_-$ and $ C_{\alpha}(v)$ is real.

  Furthermore  $c_+^t,c_-^t,c^t,C_\alpha^t(v)$  depend continuously on $t$.  In particular, if $\texttt{supp} \bar{\mu}$ preserves the proper convex cone $\mathscr{C} \subset V$ and $\tau_0(\mathscr{C})=1$, then the values of $c^t_+-c^t_-$ for $t\in [0,1]$ fill the interval $[-c_+^0,c_+^0] $. If $\alpha<1$ and $ \texttt{supp} \mu$ preserves the proper convex cone $\mathscr{X}\subset V $, then the limiting law of $ n^{-1/\alpha} S_n^x$ is supported on $\mathscr{X} $.
 \end{cor}

\begin{proof}
  The assertions for $\Lambda(t),c(t), c_+(t), c_-(t)$ follow directly from the definitions.  The formula $C_{\alpha}(v) =\alpha m_{\alpha} \Delta_v(\widetilde{\Lambda}^1)  $ implies $ C_{\alpha }(v,t) =t^{\alpha} C_{\alpha} (v)$.
  If the law of $Q$ is symmetric, the formula  $ R=Q+ \sum_{k=1}^{\infty} M_1\cdots M_k Q_{k+1}$ implies the symmetry of the law $\eta$ of $R$, hence $ \Lambda= \lim_{x\rightarrow 0_+} x^{-\alpha} (x.\eta)$ is also symmetric; it follows that $c_+=c_-$ and $C_{\alpha}(v) $ is real.

 Since $ \tau_t$ depends continuously on $t$ in the weak topology of order $\alpha$, Proposition \ref{rwgpro5-2} implies that $c^t, c_+^t, C^t_{\alpha}(v)$ depend continuously  on $t\in [0,1]$.  If $ \tau_0 (\mathscr{C}) =1$ and $ t=0$, then $ \texttt{supp} (\tau_0 \otimes \bar{\mu}) $ preserves $\mathscr{C}$, hence $c_+^0 >0$, $c_-^0=0$  and $ c_+^0-c_-^0 >0$, $c^1_+-c_-^1=-c_+^0 <0$. Then the continuity of $ c_+^t-c_-^t$ implies that all values in the interval $[-c_+^0,c_+^0] $ are taken by $ c_+^t-c_-^t$. If $ \texttt{supp} \mu$ preserves the cone $\mathscr{X}$ and $x\in \mathscr{X}$, we have $X_n^x\in \mathscr{X}$, hence by convexity $ n^{-1/\alpha} S_n^x \in \mathscr{X}$.
 Then, if $\alpha <1$, the limiting law of $ n^{-1/\alpha} S_n^x $  given by Theorem \ref{rwgTh1-5} is supported by $\mathscr{X} $.

%
%
%

\end{proof}
  In order to illustrate Theorem \ref{rwgTh1-5}, we consider,  as in \cite{Guivarch06DS}, the following
example where $d=2$, $ \mu= p \delta_h + p' \delta_{h'} $,  and $ 0<p<1$,
$h= \rho \left(                                                                                     \begin{array}{cc}
    \cos \theta & -\sin \theta  \\
     \sin \theta & \cos \theta  \\
  \end{array}
\right)
$, $ h'=\left[\left(
       \begin{array}{cc}
          \lambda & 0  \\
          0  &  \lambda' \\
       \end{array}
     \right), b \right]
$
  with  $ \theta \not \in  \mathbb{Q} \pi$, $ \rho>0$, $ 0<\lambda'<1<\lambda$, $b\neq 0 $.
Then $s_\infty=\infty$, $ \log \kappa(s)$ is convex on $[0,\infty[$ and if $ \rho$  is sufficiently small, $ L( \bar{ \mu})= \kappa'(0)<0$. Since $h'$ is proximal and $h$  is an irrational similarity, condition $i$-$p$ is satisfied by $ [\texttt{supp}\bar{\mu}]$. Since $ \theta \not \in  \mathbb{Q} \pi $, the limit set of
$[\texttt{supp} \bar{\mu}] $ is equal to $\mathbb{S}^1$ and we are in case I of Proposition \ref{rwgpro5-1}. If $\alpha \in [0,2]  $ with $ \alpha \neq 1$, we get that the limit law of the normalized Birkhoff sum  is symmetric and has Fourier transform $e^{\alpha m_\alpha c r_\alpha c^*(v)}$, where $c>0$, $ r_\alpha<0$ and $ c^*(v)= \norm{v}^{\alpha}c^*(\overline{v})$ is positive for $v \neq 0$.

If $\alpha=1$, the corresponding limit law is of Cauchy type, with Fourier transform
$e^{c m_1 r_1 \norm{v}c^*(\overline{v})}$, where $ c m_1 r_1 < 0$, $ c^*(\overline{v}) > 0$.
\appendix
\section*{Acknowledgements}
We are grateful to D. Buraczewski for useful remarks and information.  Thanks are due to an anonymous referee for careful reading of the original manuscript and for  helpful suggestions.

\end{document}